\documentclass[11pt]{amsart}
\usepackage{amsmath}
\usepackage{amssymb}
\usepackage{mathpazo}
\usepackage{eucal}
\usepackage{amsthm}
\usepackage{amscd}
\usepackage{hyperref}
\usepackage{soul}
\usepackage{mathrsfs}
\usepackage{url}
\usepackage{skull}
\usepackage{adjustbox}

\usepackage{tikz}
\usetikzlibrary{shapes.geometric, arrows}
\usetikzlibrary{shapes,arrows}
\usepackage[latin1]{inputenc}
\usepackage{tikz}
\usetikzlibrary{shapes,arrows}
\usetikzlibrary{shapes.multipart}
\usepackage{verbatim}
\usepackage{tkz-berge}
\usepackage{tikz-qtree}

\usepackage{amssymb,}
\usepackage[]{amsmath, amsthm, amsfonts,graphicx, amscd,}
\usepackage[all,cmtip]{xy}
\input amssym.def \input amssym
\usepackage{forest}
\usepackage{mdframed}
\usepackage{framed}

\begin{document}

\newtheorem{thm}{Theorem}[section]
\newtheorem{cor}[thm]{Corollary}
\newtheorem{claim}[thm]{Claim}
\newtheorem {fact}[thm]{Fact}
\newtheorem{con}[thm]{Conjecture}

\newtheorem*{thmstar}{Theorem}
\newtheorem{prop}[thm]{Proposition}
\newtheorem*{propstar}{Proposition}
\newtheorem {lem}[thm]{Lemma}
\newtheorem*{lemstar}{Lemma}
\newtheorem{conj}[thm]{Conjecture}
\newtheorem{question}[thm]{Question}
\newtheorem*{questar}{Question}
\newtheorem{ques}[thm]{Question}
\newtheorem*{conjstar}{Conjecture}
\newtheorem{fct}[thm]{Fact}
\theoremstyle{remark}
\newtheorem{rem}[thm]{Remark}
\newtheorem{exmp}[thm]{Example}
\newtheorem{cond}[thm]{Condition}
\newtheorem{np*}{Non-Proof}
\newtheorem*{remstar}{Remark}
\theoremstyle{definition}
\newtheorem{defn}[thm]{Definition}
\newtheorem*{defnstar}{Definition}
\newtheorem{exam}[thm]{Example}
\newtheorem*{examstar}{Example}
\newtheorem{assump}[thm]{Assumption}
\newtheorem{Thm}[thm]{Theorem}

\def \ta {\tau_{\mathcal{D}/\Delta}}
\def \D {\Delta}
\def \DD {\mathcal D}

\newcommand{\gen}[1]{\left\langle#1\right\rangle}

\newcommand{\pd}[2]{\frac{\partial #1}{\partial #2}}
\newcommand{\td}{\text{tr.deg.}}
\newcommand{\pp}{\partial }
\newcommand{\pdtwo}[2]{\frac{\partial^2 #1}{\partial #2^2}}
\newcommand{\od}[2]{\frac{d #1}{d #2}}
\def\Ind{\setbox0=\hbox{$x$}\kern\wd0\hbox to 0pt{\hss$\mid$\hss} \lower.9\ht0\hbox to 0pt{\hss$\smile$\hss}\kern\wd0}
\def\Notind{\setbox0=\hbox{$x$}\kern\wd0\hbox to 0pt{\mathchardef \nn=12854\hss$\nn$\kern1.4\wd0\hss}\hbox to 0pt{\hss$\mid$\hss}\lower.9\ht0 \hbox to 0pt{\hss$\smile$\hss}\kern\wd0}
\def\ind{\mathop{\mathpalette\Ind{}}}
\def\nind{\mathop{\mathpalette\Notind{}}}
\numberwithin{equation}{section}

\def\id{\operatorname{id}}
\def\Frac{\operatorname{Frac}}
\def\Const{\operatorname{Const}}
\def\spec{\operatorname{Spec}}
\def\span{\operatorname{span}}
\def\exc{\operatorname{Exc}}
\def\Div{\operatorname{Div}}
\def\cl{\operatorname{cl}}
\def\mer{\operatorname{mer}}
\def\trdeg{\operatorname{trdeg}}
\def\ord{\operatorname{ord}}

\newcommand{\m}{\mathbb }
\newcommand{\mc}{\mathcal }
\newcommand{\mf}{\mathfrak }
\newcommand{\is}{^{p^ {-\infty}}}
\newcommand{\QQ}{\mathbb Q}
\newcommand{\fh}{\mathfrak h}
\newcommand{\CC}{\mathbb C}
\newcommand{\RR}{\mathbb R}
\newcommand{\ZZ}{\mathbb Z}
\newcommand{\tp}{\operatorname{tp}}
\newcommand{\SL}{\operatorname{SL}}

\title{Ax-Lindemann-Weierstrass with derivatives and the genus 0 Fuchsian groups}
\author[G. Casale]{Guy Casale}
\address{Guy Casale, Univ Rennes, CNRS, IRMAR-UMR 6625, F-35000 Rennes, France}
\email{guy.casale@univ-rennes1.fr}

\author[J. Freitag]{James Freitag}
\address{James Freitag, University of Illinois Chicago, Department of Mathematics, Statistics,
and Computer Science, 851 S. Morgan Street, Chicago, IL, USA, 60607-7045.}
\email{freitag@math.uic.edu}

\author[J. Nagloo]{Joel Nagloo}
\address{Joel Nagloo, Department of Mathematics and Computer Science\\ Bronx Community College CUNY, Bronx, NY 10453, USA.}
\email{joel.nagloo@bcc.cuny.edu}

\thanks{G. Casale is partially supported by Math-AMSUD project ``Complex geometry and foliations''. J. Freitag is partially supported by NSF grant DMS-1700095. J. Nagloo is partially supported by NSF grant DMS-1700336. }

\maketitle
{\centering\footnotesize \it To Keiji Nishioka on his retirement.\par}

\begin{abstract} We prove the Ax-Lindemann-Weierstrass theorem with derivatives for the uniformizing functions of genus zero Fuchsian groups of the first kind. Our proof relies on differential Galois theory, monodromy of linear differential equations, the study of algebraic and Liouvillian solutions, differential algebraic work of Nishioka towards the Painlev\'e irreducibility of certain Schwarzian equations, and considerable machinery from the model theory of differentially closed fields. 

Our techniques allow for certain generalizations of the Ax-Lindemann-Weierstrass theorem which have interesting consequences. In particular, we apply our results to give a complete proof of an assertion of Painlev\'e (1895). We also answer certain cases of the Andr\'e-Pink conjecture, namely in the case of orbits of commensurators of Fuchsian groups.  
\end{abstract}

\section{Introduction}

In this paper our central work is to prove a series of functional transcendence results for the automorphic functions $j_\Gamma$ associated with a \emph{Fuchsian group} $\Gamma$ of genus 0. We will also refer to the automorphic function $j_\Gamma$ as a Hauptmodul or uniformizing function of $\Gamma$. Our general results are most easily expressed in the language of model theory and algebraic differential equations, but a \emph{special case} of our functional transcendence results is what has come to be called the Ax-Lindemann-Weierstrass theorem with derivatives for $j_\Gamma$: 

\begin{thm} \label{OGalw} Let $\m C(V)$ be an algebraic function field, where $V\subset \m A^{n}$ is an irreducible algebraic variety defined over $\m C$. Let
\[t_1,\ldots,t_n\in \m C(V)\]
take values in the upper half complex plane $\m H$ at some $P\in V$ and are geodesically independent \footnote{We say that $t_1, \ldots , t_n $ are geodesically independent if $t_i$ is nonconstant for $i =1, \ldots , n$ and there are no relations of the form $t_i = \gamma t_j$ for $i \neq j$, $i,j \in \{1, \ldots , n \}$ and $\gamma$ is an element of the commensurator of $\Gamma$.}. Then the $3n$-functions
\[j_{\Gamma}(t_1),j'_{\Gamma}(t_1),j''_{\Gamma}(t_1)\ldots,j_{\Gamma}(t_n),j'_{\Gamma}(t_n),j''_{\Gamma}(t_n)\]
(considered as functions on $V(\m C)$ locally near $P$) are algebraically independent over $\m C(V)$.
\end{thm} 

One can also describe Theorem \ref{OGalw} in more geometric terms. Let $W\subset \m A^{n}(\m C)$ be an algebraic variety which has a nonempty intersection with $\m H^n.$ Theorem \ref{OGalw} precisely characterizes those varieties $W$ whose image under the automorphic function (and derivatives) applied to each coordinate: $$\bar j_{\Gamma} :(t_1, \ldots  ,t_n) \mapsto (j_{\Gamma}(t_1),j'_{\Gamma}(t_1),j''_{\Gamma}(t_1)\ldots,j_{\Gamma}(t_n),j'_{\Gamma}(t_n),j''_{\Gamma}(t_n))$$ is contained in a proper algebraic subvariety of $\m C^{3n}$. Intuitively, the function $j_ \Gamma$ is highly transcendental, so the varieties obtained in this way should be restricted to a very special class. Indeed, Theorem \ref{OGalw} says that if $\bar j_{\Gamma} (W)$ is an algebraic variety, then $W$ must have been defined by instances of relation of the form $t_i = \gamma t_j $ where $\gamma$ is an element of the commensurator of $\Gamma$, giving a very restrictive (countable) class of complex varieties coming from the image of $\bar j_\Gamma$.

As we will explain in additional detail below, our methods also allow for more general results, which are most naturally stated in the language of model theory. For instance, statements incorporating other transcendental functions on additional coordinates (such as Weierstrass $\wp$-functions and exponential functions on semi-abelian varieties) similar to Theorem 1.6 of \cite{PilaAO} will follow from our general result. 

Theorem \ref{OGalw} is a generalization of Theorem 1.6 of \cite{PilaAO} and Theorem 1.1 \cite{Pila}, in which Pila established the special case with one group $\Gamma = PSL_2 ( \m Z)$ (in \cite{PilaAO} without derivatives and later in \cite{Pila} with derivatives). 
Theorem \ref{OGalw} also overlaps nontrivially with a number of recent results, which we detail next. Note that most of the following results do not involve the derivatives of the automorphic functions in question and are mainly concerned with arithmetic groups. 
Pila and Tsimerman \cite{AgALW} generalized Theorem 1.6 of \cite{PilaAO} to the uniformizing functions associated with the moduli spaces of higher dimensional abelian varieties (their result specializes to Theorem 1.6 of \cite{PilaAO} for purposes of comparing with Theorem \ref{OGalw}). In a different direction, Pila and Tsimerman \cite{AXS} generalized Theorem 1.6 of \cite{PilaAO} to an Ax-Schanuel type statement for the $j$-function. In \cite{ALWcom}, Ullmo and Yafaev prove an Ax-Lindemann-Weierstrass result for the uniformizing functions of cocompact Shimura varieties (without derivatives), and so a statement of Theorem \ref{OGalw} without derivatives in the case that $\Gamma$ is arithmetic and cocompact is a consequence of their work. Later, Klinger, Ullmo, and Yafaev \cite{HyperALW} removed the assumption of cocompactness, and Gao \cite{mixedALW} generalized the result to mixed Shimura varieties. Finally, Mok, Pila, and Tsimerman \cite{ASmpt} have established the (more general) Ax-Schanuel theorem with derivatives for the uniformizing function of a Shimura variety. 

The previous Ax-Lindemann-Weiestrass (ALW for short) results discussed above employ various techniques from group theory, complex variables, and number theory, but each one also shares a common element in their approach: a tool called \emph{o-minimality} originating in model theory. The theory of o-minimality is a natural generalization of real algebraic geometry to include certain non-oscillatory transcendental functions. It was developed starting in the 1980s by model theorists \cite{tameT}, but in the early 2000s, o-minimality was connected with various aspects of number theory in part through the work of Pila and Wilkie \cite{PW} and Peterzil and Starchenko \cite{PS1, PS2, PS3}. The counting theorem of Pila-Wilkie has precursors coming from number theory before the connection to o-minimality was made. See, for instance, the work of Bombieri and Pila \cite{bombpila} and the related manuscript of Sarnak \cite{Sarnak}. Diophantine properties of definable sets in o-minimal structures had also been previously investigated by Wilkie \cite{Wilkie}. However, following the Pila-Wilkie theorem, there has been an explosion of work. In \cite{PZ}, the Pila-Wilkie theorem was employed by Pila and Zannier to give a new proof of the Manin-Mumford conjecture. The strategy was immediately taken up by Masser and Zannier \cite{masser2010torsion} to prove a a special case of Pink's relative Manin-Mumford conjecture, while Pila \cite{PilaIMRN} gave new proofs of results of a Manin-Mumford-Andr\'e-Oort flavor.

The common line of reasoning in the results mentioned in the previous several paragraphs is to embed the problem in an o-minimal context by proving that the a certain analytic function (restricted to an appropriate fundamental domain) is interpretable in $\m R_{an, exp}$, an o-minimal structure in which the definable sets are given by inequalities built from the algebraic functions, the exponential function, and real analytic functions restricted to bounded sets. Following this, variants of the Pila-Zannier strategy or definable versions of results from complex geometry \cite{PS4} are used to detect and characterize algebraic relations.

Our approach is completely different, and \emph{does not} employ the theory of o-minimality at all. Rather, our proof relies on differential Galois theory, monodromy, the study of algebraic and Liouvillian solutions to linear differential equations, differential algebraic work of Nishioka towards the Painlev\'e irreducibility of certain Schwarzian equations, and considerable machinery from the model theory of differentially closed fields.  

Recently there has been a surge in interest around functional transcendence statements of the type in Theorem \ref{OGalw}, in part due to their connection with a class of problems from number theory called \emph{special points conjectures} or problems of \emph{unlikely intersections}; in \cite{PilaAO} the Ax-Lindemann-Weierstrass theorem is central to the proof of the Andr\'e-Oort conjecture for $\m C^n$. Each of the other functional transcendence results mentioned above can be applied in certain special points settings. For instance, in \cite{apAS} Daw and Ren give applications of the Ax-Schanuel conjecture proved in \cite{ASmpt}. Our functional transcendence results are no exception - we apply them to certain cases of a special points conjecture called the Andr\'e-Pink conjecture, following Orr \cite{Orr1, Orr2}. Numerous variations on the conjecture are possible (depending for instance, on the definition of Hecke-orbits one takes), but we will describe the specific setup next.   

Let $V$ be a connected Shimura variety with (connected) Shimura datum $(G,X^+)$ such that $V=\Gamma \setminus X^+$, for some congruence subgroup $\Gamma\subset G(\mathbb{Q})$ that stabilizes $X^+$. The Andr\'e-Pink conjecture predicts that when $W$ is an algebraic subvariety of $V$ and $S$ is the orbit of the commensurator of $\Gamma$, $\text{Comm}(\Gamma)$, on a point $\bar a = (a_1, \ldots , a_n)$, if $W \cap S$ is Zariski dense in $W$, then $W$ is of a \emph{very restrictive form}, which we will refer to as $\Gamma$-special, which we describe next.

 Let $j_\Gamma:X^+\rightarrow V$ be a uniformisation map. When $\gamma \in \text{Comm}(\Gamma),$ it turns out that $(j_\Gamma (t), j_\Gamma (\gamma t))$ are algebraically dependent and lie on an irreducible curve given by the vanishing of a polynomial in two variables which we will refer to as a \emph{$\Gamma$-special polynomial}. The $\Gamma$-special varieties are intersections of zero sets of $\Gamma$-special polynomials and relations of the form $x_i = b_i$ where $b_i$ is in the $\text{Comm}(\Gamma)$-orbit of $a_i$. Orr \cite{Orr1, Orr2} proved various special cases of the conjecture (for instance, when $W$ is an algebraic curve). In \cite{FreSca} Freitag and Scanlon used Pila's ALW with derivatives theorem from \cite{Pila} to prove the Andr\'e-Pink conjecture when $\bar a$ is assumed to be a transcendental point and $\Gamma$ is commensurable with $PSL_2(\m Z)$. In this paper, we generalize that result to allow for an arbitrary Fuchsian group $\Gamma .$ 

The central idea employed is a beautiful technique which has its origins in the work of Hrushovski \cite{ML1} and Buium \cite{Buium}. In order to understand intersections of algebraic varieties with an arithmetically defined set of points (e.g. torsion points on an algebraic group, Hecke orbits, etc.), replace the arithmetic set with a more uniformly defined algebraic object, the solution set of some algebraic differential differential or difference equation. 

We replace our arithmetic objects (the orbits of the commensurators of some discrete groups, $\Gamma$) by the solution sets of certain differential equations satisfied by the uniformizing functions $j_\Gamma$. An inherent restriction of the technique is that it generally only works for diophantine problems in function fields, hence the assumption that $\bar a$ is a tuple of transcendentals.  In pursuing our approach to the Andr\'e-Pink conjecture, it becomes necessary to prove more far reaching functional transcendence results than the ALW theorem as stated above; our results are most naturally phrased in terms of the \emph{model theory of differential fields}, one of the main tools we use to establish our results. One of the chief advantages of this approach is that it leads to an \emph{effective} solution of our case of Andr\'e-Pink, that is, we are able to give bounds on the degree of the Zariski closure of the intersection of $\text{Comm}(\Gamma_i)$-orbits with a variety $V$, which depend on algebro-geometric invariants of the variety $V$. So, for instance, if the variety $V$ is a non-special curve (or a variety which does not contain a special curve), we can give a bound on the number of special points contained in the curve.

At the relevant sections of our paper (e.g. \ref{algreltriv}) we will give equivalent formulations in algebro-geometric language of the model-theoretic properties we describe next. We prove that for any Fuchsian group $\Gamma$, the set defined by the differential equation satisfied by the uniformizing function $j_\Gamma$ is strongly minimal and has geometrically trivial forking geometry. This result generalizes work of \cite{FreSca} which covers the cases when $\Gamma$ is commensurable with $PSL_2(\m Z)$. In particular, our work gives many new examples of geometrically trivial strongly minimal sets in differentially closed fields. This also establishes an interesting new connection between two important dividing lines on the logic and group theory: the differential equation satisfied by $j_\Gamma$ is $\aleph_0$-categorical if and only if the group $\Gamma$ is not arithmetic. Further, we characterize all instances of nonorthogonality between these sets (each such instance comes from commensurability of two groups $\Gamma_1$ and $\Gamma_2$). These results also have various interesting consequences related to determining the isomorphism invariants of differentially closed fields, which we will not explore further in this article. 

We should also mention that this work gives a complete proof of an assertion of Painlev\'e \cite[Page 519]{pain}, concerning the \emph{irreducibility} of the differential equations satisfied by $j_\Gamma$ for $\Gamma$ a Fuchsian group. Irreducibility is closely related to the \emph{strong minimality} of a differential equation, a connection pointed out in detail by Nagloo and Pillay \cite{NagPil}. The original definition of irreducibility applies to nonlinear differential equations and was given by Painlev\'e \cite[pages 490-496]{pain}. A definition (for functions) using more modern language was given by Umemura, for instance see \cite[pages 754-755]{Okamato}. There have been claims (usually non-specific) that Painlev\'e's definition is not completely rigorous. For instance, see the third paragraph of page 755 of \cite{Okamato} and page 772 of \cite{Umemura1}. These claims seem to originate with Umemura \cite{Umemura0}, however the only specific complaint with Painlev\'e's definitions which we find there is related to some subtleties around algebraic and analytic groups (for instance, see pages three and eight). Similar points are made also in \cite{Umemura-1}. These complaints seem mainly to affect some proofs of results from \cite{pain}, but not directly the \emph{definition} of irreducibility.

In \cite{NishiokaI} and \cite{NishiokaII}, Nishioka proved a weak form of Painlev\'e's assertion; various techniques from Nihsioka's paper have inspired our work. 

\subsection*{Acknowledgements}
G.C and J.N take this opportunity to thank the organizers of the CIRM meeting {\emph ``Algebra, Arithmetic and Combinatorics of Differential and Difference Equations''} in May 2018, where this research collaboration started. We also thank the anonymous referees for their comments and suggestions.

\section{The basic theory}\label{basic}

\subsection{Fuchsian groups and the associated Schwarzian equations}
We direct the reader to \cite{Katok} and \cite{Lehner} for the basics on Fuchsian groups and the corresponding automorphic functions. The appendices of \cite{Venkov} also give a very detailed introduction to the associated Schwarzian equations.

\par Let $\mathbb{H}$ be the upper half complex plane and let $\overline{\mathbb{H}}:=\mathbb{H}\cup{\bf P}^1(\mathbb{R})$. Recall that $SL_2(\mathbb{R})$ and $PSL_2(\mathbb{R})$ acts on $\mathbb{H}$ (and $\overline{\mathbb{H}}$) by linear fractional transformation: for 
$\begin{pmatrix}
    a & b \\
    c & d
  \end{pmatrix}
  \in SL_2(\mathbb{R})$ and $\tau\in\mathbb{H}$
\[ \begin{pmatrix}
    a & b \\
    c & d
  \end{pmatrix}\cdot\tau=\frac{a\tau+b}{c\tau+d}.
\]
This action yields all the orientation preserving isometries of $\mathbb{H}$.
\par Let $\Gamma\subset PSL_2(\mathbb{R})$ be a Fuchsian group, that is, assume that $\Gamma$ is a discrete subgroup of $PSL_2(\mathbb{R})$. A point $\tau\in\overline{\mathbb{H}}$ is said to be a cusp if its stabilizer group $\Gamma_{\tau}=\{g\in\Gamma\;:\;g\cdot \tau=\tau\}$ has infinite order. We also assume throughout that $\Gamma$ is of first kind ({\it i.e.,} its limit set is ${\bf P}^1(\mathbb R)$) and of genus zero\footnote{The methods of proof and results of the current article can be generalized with additional effort to the case of arbitrary genus. This will be tackled in a forthcoming work of the authors along with D. Bl\'asquez-Sanz around Ax-Schanuel Theorems for Fuchsian functions.} ({\it i.e.,} $\Gamma \setminus \mathbb H$
can be compactified to a compact Riemann surface of genus 0, cf. paragraph after Example \ref{exmp1}). For any point $\tau\in\mathbb{H}$, the group $\Gamma_{\tau}$ is finite and cyclic. A point $\tau\in\mathbb{H}$ is said to be elliptic of order $\ell\geq 2$ if $|\Gamma_{\tau}|=\ell$. Our assumptions on $\Gamma$ ensure that modulo $\Gamma$ there are only finitely many orbits under $\Gamma$ of elliptic points. If $m_1,\ldots,m_r$ denotes the orders of the elliptic points as well as of those of the cusps (which would be $\infty$'s), then  $\Gamma$ is said to have signature $(0;m_1,\ldots,m_r)$. The zero here reflects that $\Gamma$ has genus $0$. The group then has the following presentation
\[\Gamma=\gen{g_1,\ldots,g_r\;:\;g_1^{m_1}=\ldots=g_r^{m_r}=g_1\cdots g_r=I}\] 
\noindent When one or more of the $m_i$'s are infinity, one simply remove the relations containing the infinite $m_i$'s in the above presentation.

\begin{exmp}\label{exmp1}
$PSL_2(\mathbb{Z})$ is a Fuchsian (triangle) group of type $(0;2,3,\infty)$.  Recall that traditionally we might consider the following generators of  $SL_2(\mathbb{Z})$:
\[ T=\begin{pmatrix}
    1 & 1 \\
    0 & 1
  \end{pmatrix} \;\;\;,\;\;\; S=\begin{pmatrix}
    0 & -1 \\
    1 & 0
  \end{pmatrix}.\]
  Nonetheless, by setting $g_1=-S$, $g_2=-T^{-1}S$ and $g_3=T$ one has that
\[SL_2(\mathbb{Z})=\gen{g_1,g_2,g_3\;:\;g_1^{2}=g_2^3=g_1g_2g_3=-I}.\]
  $PSL_2(\mathbb{Z})$ is obtained from the above using the natural projection $\pi:SL_2(\mathbb{R})\rightarrow PSL_2(\mathbb{R})$.
\end{exmp}

As is well known, $\Gamma$ acts on the set $C_{\Gamma}$ of its cusps and the action of $\Gamma$ on $\mathbb{H}_{\Gamma}:=\mathbb{H}\cup C_{\Gamma}$, yields a compact Riemann surface $\Gamma\setminus \mathbb{H}_{\Gamma}$ or equivalently a projective non-singular curve $X(\Gamma)$, which is of genus zero. The group $\Gamma$ is said to be cocompact if $C_{\Gamma}=\emptyset$. In other words, if the quotient $\Gamma\setminus \mathbb{H}$ is already a compact space. By an \emph{automorphic function} for $\Gamma$, we mean a meromorphic function $f$ on $\mathbb{H}$ which is meromorphic at every cusp of $\Gamma$  and which is invariant under the action of $\Gamma$:
\[f(g\cdot \tau)=f(\tau)\;\;\;\text{ for all }\;g\in\Gamma\text{ and } \tau\in\mathbb{H}.\]
One has that the field of automorphic functions $\mathcal{A}_0(\Gamma)$ for $\Gamma$ (or equivalently the field of meromorphic functions of $\Gamma\setminus \mathbb{H}_{\Gamma}$) is isomorphic to the field $\mathbb{C}(X(\Gamma))$ of rational functions on $X(\Gamma)$. By an Hauptmodul or uniformizer $j_{\Gamma}(t)$ for $\Gamma$ we mean an automorphic function for $\Gamma$ which generates $\mathcal{A}_0(\Gamma)$ (and so $\m C(j_{\Gamma})\simeq\mathbb{C}(X(\Gamma))$). We will also write $j_{\Gamma}$ for the biholomorphism $\Gamma\setminus \mathbb{H}_{\Gamma}\rightarrow{\bf P}^1(\mathbb{C})$. Let us point out that the function $j_{\Gamma}$ is not unique. This follows from the existence of nontrivial automorphisms of the curve $X(\Gamma)$. Moreover, it is well known that the function $j_{\Gamma}$ \emph{is unique} once its values at three points have been specified.

The uniformizer $j_{\Gamma}$ also satisfies a third order ordinary differential equation of Schwarzian type:

\begin{equation}\tag{$\star$} \label{stareqn}
S_{\frac{d}{dt}}(y) +(y')^2\cdot R_{j_{\Gamma}}(y) =0
\end{equation}

where $S_{\frac{d}{dt}}(y)=\left(\frac{y''}{y'}\right)' -\frac{1}{2}\left(\frac{y''}{y'}\right)^2$ denotes the Schwarzian derivative ($'=\frac{d}{dt}$) and $R_{j_{\Gamma}}\in\mathbb{C}(y)$ depends on the choice of $j_{\Gamma}$. Moreover, the `shape' of the function $R_{j_{\Gamma}}$, depends on knowing the fundamental half domain for the $\Gamma$-action on $\mathbb{H}$: Let us assume that it is given by a polygon $P$ with $r$ vertices $b_1,\dots,b_r$ and whose sides are identified by pairs and having internal angles
$\alpha_1\pi,\ldots,\alpha_r\pi$. Then
\[R_{j_{\Gamma}}(y)=\frac{1}{2}\sum_{i=1}^{r}{\frac{1-\alpha_i^2}{(y-a_i)^2}}+\sum_{i=1}^{r}{\frac{A_i}{y-a_i}}\]

where $j_{\Gamma}(b_i)=a_i$ and the $A_i$'s are real numbers that do not depend on $j_{\Gamma}$ and satisfy some very specific algebraic relations (cf. \cite[page 142]{Venkov}). 

\begin{exmp}\label{jfunction} A well-known example is $\Gamma=PSL_2(\mathbb{Z})$ and $j_{\Gamma}$ is the classical $j$-function. In this case the equation is given with
\[
R_{j}(y)=\frac{y^2-1968y+2654208}{y^2(y-1728)^2}
\]
$\Gamma = PSL_2 (\m Z)$ is an example of a \emph{triangle group}. In the appendix the case of the Fuchsian triangle groups is explained in more details. We also direct the reader to \cite{BayTra1} where more examples of uniformizers - beyond those attached to triangle groups - are studied.
\end{exmp}

There is a long tradition of functional transcendence results around automorphic functions. For instance, a very weak form of our results was conjectured by Mahler, and answered by Nishioka:
\begin{fct}[\cite{Nish}]\label{nishioka1}
The Hauptmodul $j_{\Gamma}$ satisfies no algebraic differential equation of order two or less over $\mathbb{C}(t,e^{ut})$, for any $u\in \m C$. The same is true for all $\Gamma$-automorphic functions.
\end{fct}
Using the Seidenberg's embedding theorem and the composition rule of the Schwarzian derivative, we also have
\begin{lem}[cf. \cite{FreSca}]\label{fct1}
Let $K$ be an abstract differential field extension of $\mathcal{C}(t)$ generated by $y_1, \ldots , y_n$ solutions of equation (\ref{stareqn}). Here $\mathcal{C}$ is a finitely generated subfield of $\m C$. Then there are elements $g_1, \ldots , g_n \in GL_2(\m C)$ such that $$K \cong\mathcal{C} \langle t, j_\Gamma (g_1 t) , \ldots , j_\Gamma ( g_n t) \rangle.$$ 
\end{lem}
 \begin{proof} 
By Seidenberg's embedding theorem, we may assume that $y_1, \ldots , y_n$ are meromorphic functions on some domain $U$ contained in $\m H$. Since the $j_{\Gamma}$ is a non constant holomorphic function from $\m H$ to $\m C$, there are holomorphic functions $\psi_i:U\rightarrow{\m H}$, such that $y_i(t)=j_{\Gamma}(\psi_i(t))$. Repeating the arguments in \cite{FreSca} - using the composition rule for $S_{\frac{d}{dt}}(y)$ and the fact that $j_{\Gamma}(\psi_i(t))$ is a solution of the equation (\ref{stareqn}) - we get that $S_{\frac{d}{dt}}(\psi_i(t))=0$. Hence $\psi_i(t)=g_it$ for some $g_i\in GL_2(\mathbb{C})$.
 \end{proof} 
\begin{rem}
Notice that the $g_1, \ldots , g_n$ are not arbitrary elements of $GL_2(\m C)$. Indeed, since the $y_i(t)$'s are meromorphic on $U\subset \m H$, it must be that $g_i:U\rightarrow \mathbb{H}$. Also, for each $i$, from the inverse $g^{-1}_i$ of $g_i$ we have well defined solutions $j_\Gamma (g^{-1}_i t)$ and $j_\Gamma (g_jg^{-1}_i t)$ of (\ref{stareqn}).
\end{rem}

In this paper, depending on the context, we will freely alternate between thinking of solutions of the Schwarzian equation (\ref{stareqn}) as points in an abstract differential field or as meromorphic functions of the form $j_\Gamma (gt)$. The latter form will always mean that $g$ is an element of $GL_2(\m C)$ that maps (a subset of) $\mathbb{H}$ to $\mathbb{H}$.

\subsection{Arithmetic Fuchsian groups} 

We have already seen one important dividing line among those $\Gamma$, which we consider, namely whether or not $\Gamma$ is cocompact. Another, perhaps even more important (for our work) property that $\Gamma$ might possess is that of \emph{arithmeticity}. We will begin by reviewing some key definitions. A standard reference for this subsection is \cite{Vigneras}. Throughout $\Gamma\subset PSL_2(\mathbb{R})$ is a Fuchsian group of first kind of genus zero.
\par Let $F$ be a field of characteristic zero and let $A$ be a quaternion algebra over $F$: a central simple algebra of dimension $4$ over $F$. Since the characteristic of $F$ is zero, there are elements $i$ and $j$ in $A$ and $a,b\in F^*$ such that
\[i^2 =a,\;\;\; j^2=b,\;\;\; ij=-ji,\]
and $A=F+Fi+Fj+Fij$. As customary, we use the Hilbert symbol notation $A=\left(\frac{a,b}{F}\right)$.  For $\alpha=a_0+a_1i+a_2j+a_3ij\in A$, we define its conjugation as $\overline{\alpha}=a_0-a_1i-a_2j-a_3ij\in A$. Then, the reduced trace $tr(\alpha)$ is defined to be $\alpha+\overline{\alpha}=2a_0\in F$ and the reduced norm $n(\alpha)$ is defined to be $\alpha\overline{\alpha} =a _0^2-a_1^2a-a_2^2b+a_3^2ab\in F$.

\begin{exmp}
For example, the $2\times 2$ matrices over $F$ is given by $M_{2}(F)=\left(\frac{1,1}{F}\right)$ and in this case the norm is simply the determinant.
\end{exmp}

If $F=\mathbb{R}$ or a non-Archimedean local field, then up to isomorphism, there are only two quaternion algebras: $M_{2}(F)$ or a division algebra. When $F$ is a number field and $v$ a place of  $F$, we say that $A$ splits at $v$ if the localization $A\otimes_F F_v$ is isomorphic to $M_{2}(F_v)$. Here $F_v$ denote the completion of $F$ with respect to $v$. If on the other hand $A\otimes_F F_v$ is isomorphic to a division algebra, we say $A$ ramifies at $v$. It is known that the number of ramified places is finite and the discriminant of $A$ is defined as the product of the finite ramified places.
\par Assume now that $F$ is a totally real number field of degree $k+1$ and we denote by $\mathcal{O}_F$ its ring of integers. Assume further that $A$ splits at exactly one infinite place, that is,
\[A\otimes_{\mathbb{Q}} \mathbb{R}\simeq M_{2}(\mathbb{R})\times \mathcal{H}^k\]
where $\mathcal{H}$ is Hamilton's quaternion algebra $\left(\frac{-1,-1}{\mathbb{R}}\right)$. Then, up to conjugation, there is a unique embedding $\rho$ of $A$ into $M_{2}(\mathbb{R})$. In particular for any $\alpha\in A$, one has that $n(\alpha)=det(\rho(\alpha))$.
\par Let $\mathcal{O}$ be an order in $A$, namely a finitely generated $\mathcal{O}_F$-module that is also a ring with unity containing a basis for
$A$, that is $\mathcal{O}\otimes_{\mathcal{O}_F} F\simeq A$. Denote by $\mathcal{O}^1$ the norm-one group of $\mathcal{O}$, that is $\mathcal{O}^1=\{\alpha\in \mathcal{O}\;:\;n(\alpha)=1\}$. Then the image $\rho(\mathcal{O}^1)$ of $\mathcal{O}^1$ under $\rho$ is a discrete subgroup of $SL_2(\mathbb{R})$. We denote by $\Gamma(A,\mathcal{O})$  the projection in $PSL_2(\mathbb{R})$ of the group $\rho(\mathcal{O}^1)$.

\begin{defn}
The group $\Gamma$ is said to be arithmetic if it is commensurable with a group of the form $\Gamma(A,\mathcal{O})$.
\end{defn}

Perhaps the best known example of an arithmetic group is $PSL_2(\mathbb{Z})$. Recall that two groups $\Gamma_1$ and $\Gamma_2$ are commensurable, denoted by $\Gamma_1\sim\Gamma_2$,  if their intersection $\Gamma_1\cap\Gamma_2$ has finite index in both $\Gamma_1$ and $\Gamma_2$.

If $\Gamma$ is arithmetic, then the quotient $\Gamma\setminus \mathbb{H}$ is called a {\it Shimura curve}. In this article, by abuse of terminology we will refer to $\Gamma\setminus \mathbb{H}$ as a Shimura curve of genus $g$ if and only if $\Gamma\setminus \mathbb{H}_{\Gamma}$ is of genus g, and we are interested solely in the case where $g = 0$. As is well known, Shimura curves are generalizations of classical modular curves. We direct the reader to \cite{BayTra} and \cite{Tu} where the Schwarzian equations for many examples of these curves are derived and studied.

We now look at the connection between arithmeticity of $\Gamma$ and existence of correspondences on ${\bf P}^1(\mathbb{C})\times{\bf P}^1(\mathbb{C})$ whose preimage under $j_\Gamma$ is also algebraic (cf. \cite{Mochi} and \cite{Shim}). Let $\text{Comm}(\Gamma)$ be the commensurator of $\Gamma$, namely
\[\text{Comm}(\Gamma)=\{g\in PSL_2(\mathbb{R})\;:\; g\Gamma g^{-1}\sim\Gamma\}.\]
By a $\text{Comm}(\Gamma)$-correspondence on ${\bf P}^1(\mathbb{C})\times{\bf P}^1(\mathbb{C})$ we mean a subset of the form
\[X(\Gamma g \Gamma)=\{j_{\Gamma}(\tau)\times j_{\Gamma}(g\cdot\tau):\tau\in\mathbb{H}_{\Gamma}\}\]
where $g\in \text{Comm}(\Gamma)$. It turns out that $X(\Gamma g \Gamma)$ is an absolutely irreducible curve and that it depends only on the coset $\Gamma g \Gamma$ and not on the choice of $g$ (cf. \cite{Shim} Chapter 7). We suppose that $X(\Gamma g \Gamma)$ is given by the equation $\Psi_{\tilde{g}}(X,Y)=0$, so that $\Psi_{\tilde{g}}(j_{\Gamma},j_{\Gamma}(gt))=0$. We write $\tilde{g}$ to highlight that the equation depends on $\Gamma g \Gamma$ and not $g$. With this notation, for $g_{1},g_{2}\in GL_2(\mathbb{C)}$ we more generally say that $j_{\Gamma}(g_{1}t)$ and $j_{\Gamma}(g_{2}t)$ are in $\text{Comm}(\Gamma)$-correspondence if $\Psi_{\tilde{g}}(j_{\Gamma}(g_{1}t), j_{\Gamma}(g_{2}t))=0$ for some $\Gamma g \Gamma$. One has the following result of Margulis:

\begin{fct}[\cite{Marg}]
The group $\Gamma$ is arithmetic if and only if $\Gamma$ has infinite index in $\text{Comm}(\Gamma)$ and as a result there are infinitely many $\text{Comm}(\Gamma)$-correspondences.
\end{fct}

The modular polynomials (also known as Hecke correspondences) are the classical examples (when $\Gamma = PSL_2( \m Z)$). Returning to the Schwarzian equations we see that arithmetic Fuchsian groups of genus $0$ give examples of ODE's with rich binary relations.

\subsection{A touch of Model theory} \label{modelstuff} 

We end this section by saying a few words about the concepts in model theory and differential algebra that will be required in the next sections. We will then be ready to state the main results in the paper. Throughout we work in a differentially closed field of characteristic zero. 

\begin{defn}
A definable set $\mathcal{Y}$ is said to be strongly minimal if it is infinite and every definable subset is finite or co-finite.
\end{defn}

\begin{rem}\label{SMspecial} 
Let $\mathcal{Y}$  be defined by an ODE of the form $y^{(n)}=f(t,y,y',\ldots,y^{(n-1)})$, where $f$ is rational over $\mathbb{C}(t)$ (this is of course the case for the sets defined by the Schwarzian equations). Then $\mathcal{Y}$ is strongly minimal if and only if for any differential field extension $K$ of $\mathbb{C}$ and solution $y\in\mathcal{Y}$ , $\text{tr.deg.}_KK\gen{y}=0$ or $n$. 
\end{rem}

Strong minimality is fundamental to the model theoretic approach to differential algebra (cf. \cite{NagPil}). It is also closely related to the Painlev\'e notion of  irreducibility of the ODE with respect to classical functions \cite{umemura}. It turns out that there is a very general classification of strongly minimal sets in differentially closed fields about which we will say a few more words in Section \ref{trivial}. For now, we only mention the kind of strongly minimal set that is relevant for equation (\ref{stareqn}):

\begin{defn}
Let $\mathcal{Y}$ be an $F$-definable  strongly minimal  set. Then $\mathcal{Y}$ is \emph{geometrically trivial} if for any differential field extension $K$ of $F$, and for any distinct solutions $y_{1},\ldots,y_{m}$, if the collection consisting of $y_{1},\ldots,y_{m}$ together with all their derivatives $y_{i}^{(j)}$ is algebraically dependent over $K$ then for some $i<j$, $y_{i}, y_{j}$ together with their derivatives are algebraically dependent over $K$. 
\end{defn}

So geometric triviality limits the complexity of the structure of the algebraic relations on the definable set. However, given such a set, for the results which we pursue, much greater precision is required. Throughout for simplicity, we will say that an ODE is strongly minimal and geometrically trivial just in the case that its solution set is strongly minimal as a definable set. Our first theorem is the following:

\begin{Thm}\label{Geotrivial}
The Schwarzian equation $(\star)$ for the Hauptmodul $j_{\Gamma}$ of a genus $0$ Fuchsian group $\Gamma$ of first kind is strongly minimal and geometrically trivial.
\end{Thm}

We will give the proof in subsection \ref{trivial}. This result was previously only known for $PSL_2(\mathbb{Z})$ (the $j$-function see Example \ref{jfunction}) as well as for arithmetic subgroups of $PSL_2(\mathbb{Z})$ (cf. \cite{FreSca}). Our proof, which handles all Schwarzian equations of genus zero Fuchsian functions at once, also is the first which does not use o-minimality. The first proof for $PSL_2(\mathbb{Z})$ (of \cite{FreSca}) relied on the main result of \cite{Pila}, where Pila employs the same strategy from \cite{PilaAO}, relying on o-minimality and counting of points of bounded height. Later, \cite{Aslanyan} also gave a proof of the special case of $PSL_2(\mathbb{Z})$ which relied on the Ax-Schanuel type results of \cite{AXS}, where again, an o-minimal strategy was employed. 

It is worth mentioning that Painlev\'e \cite[Page 519]{pain} claimed that strong minimality (or irreducibility as he called it) would hold for the equations we consider. In \cite{NishiokaII}, Nishioka proved a very weak form of that conjecture. Nevertheless, Nishioka's paper contains techniques that inspired our own proof. 
\par We have also obtained a full description of the structure of the definable sets. One can think of these results as a weak form of the Ax-Lindemann-Weierstrass Theorem with derivatives for $\Gamma$.\footnote{The ALW statement we are pursuing allows for characterizing algebraic relations between functions which don't formally satisfy the same differential equation, but we will use to Theorems \ref{ALW} and \ref{ALW bis} to prove our most general results, which imply the pertinent version of ALW.}

\begin{Thm}\label{ALW}
Suppose that $\Gamma$ is arithmetic and suppose that $j_{\Gamma}(g_{1}t),...,j_{\Gamma}(g_{n}t)$ are distinct solutions of the Schwarzian equation $(\star)$ that are pairwise not in $\text{Comm}(\Gamma)$-correspondence. Then the $3n$ functions $j_{\Gamma}(g_{1}t),j'_{\Gamma}(g_{1}t),j''_{\Gamma}(g_{1}t),\ldots,j_{\Gamma}(g_{n}t),j'_{\Gamma}(g_{n}t),j''_{\Gamma}(g_{n}t)$ are algebraically independent over ${\mathbb C}(t)$. 
\end{Thm}

\begin{Thm}\label{ALW bis}
Suppose that $\Gamma$ is non-arithmetic. Then there is a $k\in\mathbb{N}$ such that if $j_{\Gamma}(g_{1}t),...,j_{\Gamma}(g_{n}t)$ are distinct solutions of the Schwarzian equation $(\star)$ satisfying 
\[\text{tr.deg.}_{\mathbb{C}(t)}\mathbb{C}\gen{t,j_{\Gamma}(g_{1}t)\ldots,j_{\Gamma}(g_{n}t)}=3n,\]
then for all other solutions $j_{\Gamma}(gt)$, except for at most $n\cdot k$, 
\[\text{tr.deg.}_{\mathbb{C}(t)}\mathbb{C}\gen{t,j_{\Gamma}(g_{1}t)\ldots,j_{\Gamma}(g_{n}t),j_{\Gamma}(gt)}=3(n+1).\]
\end{Thm}


So, by the previous two theorems, we have that the set defined by the Schwarzian equation $(\star)$ is $\aleph_0$-categorical if and only if the group $\Gamma$ is non-arithmetic. It was a long-standing open problem in the model theory of differential fields (recently resolved by \cite{FreSca}) to find a non-$\aleph_0$-categorical geometrically trivial strongly minimal set; the non-existence of such sets was part of a strategy for certain diophantine problems suggested by Hrushovski \cite[see page 292]{HruICM}. Theorem \ref{ALW} gives many new examples of geometrically trivial non-$\aleph _0 $-categorical equations, and together with Theorem \ref{ALW bis} also provides an interesting connection between categoricity and arithmetic groups. We view the following question as the next major challenge in the classification of geometrically trivial strongly minimal sets in differentially closed fields:

\begin{ques} 
In $DCF_0$, are there non-$\aleph_0$-categorical strongly minimal sets that do not arise from arithmetic Fuchsian groups?\footnote{Later in the paper, it will be clear to model theorists that by ``arise from" arithmetic Fuchsian groups, we mean "are non-orthogonal to the differential equation (\ref{stareqn}) or one of its other fibers". An answer to the question is of interest in part because if there were a strong classification of the geometrically trivial strongly minimal sets in differential fields, some of the strategy laid out in \cite{HruICM} for certain diophantine problems might be possible.}
\end{ques} 

Finally let us talk about the full Ax-Lindemann-Weierstrass Theorem with derivatives for $\Gamma$. We closely follow the description of the problem as in \cite{Pila}.  Let $V\subset \m A^{n}$ be an irreducible algebraic variety defined over $\m C$ 
such that $V(\m C)\cap\m H^n\neq\emptyset$ and $V$ projects dominantly to each of its coordinates (each coordinate function is nonconstant). Let $t_1,\ldots,t_n$ be the functions on $V$ induced by the canonical coordinate functions on $\m A^{n}$. We say that $t_1,\ldots,t_n$ are {\it $\Gamma$-geodesically independent} if there are no
relations of the form
\[t_i=g t_j\]
where $i\neq j$ and $g\in \text{Comm}(\Gamma)$ acts by fractional linear transformations.

\begin{Thm}\label{fullALW} 
With the notation (and assumption $V(\m C)\cap\m H^n\neq\emptyset$) as above, suppose that $t_1,\ldots,t_n$ are $\Gamma$-geodesically independent. Then the 3n functions
\[j_{\Gamma}(t_1),j'_{\Gamma}(t_1),j''_{\Gamma}(t_1)\ldots,j_{\Gamma}(t_n),j'_{\Gamma}(t_n),j''_{\Gamma}(t_n)\]
(defined locally) on $V(\m C)$ are algebraically independent over
$\m C(V)$.
\end{Thm}

We will prove Theorem \ref{fullALW} in section \ref{OrthALW}. Pila \cite{Pila} had already proved the result for $PSL_2(\mathbb{Z})$ (see also \cite{FreSca} where the same is established for arithmetic subgroups of $PSL_2(\mathbb{Z})$).

\section{A criterion for Strong minimality of a general Fuchsian equation}

We now aim to give a criterion that can be used to show that the Schwarzian equation $(\star)$ is strongly minimal. This criterion is applicable to Schwarzian equations in the general sense, namely to any equation of the form
\begin{equation}\tag{$\star'$}
S_{\frac{d}{dt}}(y) +(y')^2\cdot R(y) =0\label{(*)}.
\end{equation}
So here we do not assume the rational function $R$ to necessarily correspond to some Hauptmodul. We only require that $R$ is rational over $\mathbb{C}$. By the Riccati equation attached to $(\star')$ we mean the equation
\begin{equation}\tag{$\star\star$}
\frac{du}{dy}+u^2+\frac{1}{2}R(y) =0.
\end{equation}

\begin{cond}\label{Ric}
The Riccati equation ($\star\star$) has no solution in $\mathbb{C}(y)^{alg}$.
\end{cond}

\begin{Thm} \label{juststrmin}
Let $(K,\partial)$ be any differential field extension of $\mathbb{C}$ and let us assume that Condition \ref{Ric} holds. If $j_R$ is a solution of the Schwarzian equation $(\star')$ we have that
\[\text{tr.deg.}_KK\gen{j_R}=0\text{ or }3.\] In other words, if Condition \ref{Ric} holds, then equation $(\star')$ is strongly minimal.
\end{Thm}
\begin{proof}
For contradiction, assume that there is a finitely generated differential field extension $F$ of $\mathbb C$ that witnesses non-strong minimality of the equation $(\star')$ (i.e., an order 1 or 2 $F$-differential subvariety exists). Throughout, we write $K=F(t)$ and let $j_R$ be a solution of the Schwarzian equation $(\star')$ such that $\text{tr.deg.}_KK\gen{j_R}=1\text{ or }2$, respectively.\footnote{If $f$ is a solution of equation $(\star ')$ generating a differential field extension of $F$ of transcendence degree one or two, taking $j_R$ to be a realization of a non-forking extension of the type of $f$ over $F$ to the field $K=F(t)$ gives such a solution $j_R$ of $(\star ')$.} 

Furthermore, using Seidenberg's embedding theorem we can assume that $K$ is a subfield of $\mathscr M (U)$ the field of meromorphic functions on an open domain $U \subset \mathbb C$ and that $j_R \in \mathscr M (U)$.

Let $P\in\mathbb{C}[y]$ be a denominator of the rational function $R(y)$. Let $L =K[y, \frac{1}{P(y)},y', \frac{1}{y'},y'']$ be the polynomial ring equipped with the derivation
\begin{itemize}
\item $D =  \partial + y' \frac{\partial}{\partial y} + y'' \frac{\partial}{\partial y'} +  \left( \frac{3}{2}\frac{y''^2}{y'} -(y')^3 R(y) \right) \frac{\partial}{\partial y''}$
\end{itemize}
making $L$ a universal $(K,\partial)$-algebra generated by a non constant solution of the Schwarzian equation.
One also defines an action of $\mathfrak{psl}_2(\m C)$ by : 

\begin{itemize}
\item $ X = \partial$
\item $ H = t\partial - y'\frac{\partial}{\partial y'} -2 y'' \frac{\partial}{\partial y''}$
\item $ Y = \frac{t^2}{2}\partial -ty'\frac{\partial}{\partial y'} - (2ty'' +y')\frac{\partial}{\partial y''}$
\end{itemize}

It is easily verified that $[X, H] =  X$, $[H, Y]= Y$, $[X,  Y] = H$ (the basis $\tilde X =X$, $\tilde Y =2Y$ and $\tilde H = 2H$ is a Chevalley basis, {\it i.e.,} satisfying $[\tilde X, \tilde H] =  2\tilde X$, $[\tilde H, \tilde Y]= 2 \tilde Y$, $[\tilde X,  \tilde Y] = \tilde H$).
Furthermore, the equalities $[D, X] = 0$, $[D,  H] = D$, $[D, Y] = t D$ can be easily verified.

When $K = \mathbb{C}(t)$, the algebraic group $PSL_2(\m C)$ acts on $L$ by
\begin{equation}\label{homography}
h (t,y,y',y'') = \left( \underline{h}(t), y, \frac{y'}{\underline{h}'(t)}, \frac{y''}{(\underline{h}'(t))^2} - y' \frac{\underline{h}''(t)}{(\underline{h}'(t))^3}\right)
\end{equation}
where $\underline{h}$ denotes the homography of the projective line associated to an element $h$ of $PSL_2(\m C)$. For $F \in L$, one defines $(h)^\ast D \cdot F = h \circ D \circ h^{-1}(F)$ . Direct computations give that $(h)^\ast D = \frac{1}{\underline{h}'(t)}D$. 
This equality means that the set of meromorphic solutions of a Schwarzian equation is stable by the action of $PSL_2(\CC)$ by precomposition. The previously given action of $\mathfrak{psl}_2(\mathbb C)$ is the infinitesimal action of $PSL_2(\m C)$. 

When $K \subset \mathscr M(U)$ then a fixed element $h \in PSL_2(\mathbb{C})$ maps $L$ onto an isomorphic subfield in $\mathscr{M}(h^{-1}(U))$ but the whole group does not act on $L$. Now let $I \subset L$ be the annihilator of the solution $j_R$ and $Z$ be the zero locus of $I \cap \mathscr O(U)[y,y',y'']$ in $U \times \mathbb C^3$, where $\mathscr O(U)$ is the ring of holomorphic functions on $U$. We have that $Z$ is an analytic variety, affine over $U$, and that its $K$-fibers are algebraic varieties over $K$. We have the following lemma

\begin{lem}
The dimension of the subalgebra $\mathfrak b$ of $\mathfrak{psl}_2(\m C)$ stabilizing $I$ equals the dimension of $Z$ over $K$.
\end{lem}
 
 \begin{proof}
 Let  $p \in Z$ be a smooth point in the graph of  $(j_R,j'_R, j''_R)$. Then the evaluation of $X$, $H$, $Y$ and $D$ at $p$ give a basis of $T_p(  U \times \mathbb C^3)$. If $v$ is in $T_p(Z) \subset T_p(  U \times \mathbb C^3)$ then there exists $V \in \mathfrak{psl}_2 + \mathbb C D$ whose value at $p$ is $v$. 
 
We first show that $V\cdot I \subset I$. To see this, notice that for $P \in I$ we have that $D\cdot (V\cdot P) = V\cdot (D\cdot P) +f D\cdot P$, for some $f\in\mathbb{C}(t)$. Here we use the equalities $[D, X] = 0$, $[D,  H] = D$, $[D, Y] = t D$. Since by definition $D\cdot P \in I$, we have that the ideal $J$ generated by $I$ and $V\cdot I$ is stable by $D$. Moreover all elements of $J$ vanish at $p$ thus on the whole graph of $(j_R,j'_R, j''_R)$. By maximality $V\cdot I \subset I$.

The stabilizer of $I$ in $\mathbb C X + \mathbb C H + \mathbb C Y + \mathbb C D$ has the dimension of $T_pZ$ and thus that of $Z$. Because $D$ is tangent to $Z$, the dimension of $Z$ over $K$ is the dimension of the stabilizer of $I$ in $\mathfrak{psl}_2$.
 \end{proof}
Our assumption $\text{\rm tr.deg}_KK\gen{j_R}=1\text{ or }2$ gives that the stabilizer, denoted by $\mathfrak b$, is a non-trivial proper subalgebra of $\mathfrak{psl}_2(\m C)$. Every such a proper subalgebra is contained in a $2$-dimensional Lie subalgebra of $\mathfrak{psl_2}=\mathfrak{sl_2}$. Furthermore, the group $PSL_2(\mathbb C)$ acts on $\mathfrak{psl_2}$ by the adjoint representation and under this action all Lie subagebras of $\mathfrak{psl_2}$ of dimension 2 are conjugate to one another (cf. \cite[Section 16]{Hump}).
 
 Let $g \in PSL_2(\mathbb C)$ be an element conjugating a dimension two subalgebra containing $\mathfrak{b}$ to the algebra generated by $X$ and $H$. Then $g$ acts as an homography on ${\bf P}^1(\mathbb{C})$ and transforms $K\subset \mathscr M (U)$ to $K^g\subset \mathscr M (g^{-1}(U))$.
 
The induced isomorphism of $L$ to $L^g =K^g[y, \frac{1}{P(y)},y',\frac{1}{y'},y'']$ sending $y$ to $y$, $y'$ to $y'g'(t)$ and $y''$ to $y''g'(t)^2 - y'g''(t)$ preserves $D$ up to multiplication by an element of $K$ (see equation \ref{homography}) and induced the adjoint action on $\mathfrak{psl_2}$. The transcendence degree of $j_R$ over $K$ is the transcendence degree of $j_R\circ g$ over $K^g$ but now we have ensured that the stabiliser is included in the Lie algebra generated by $X$ and $H$. Let us forget that we change the field and assume $\mathfrak b$ is included in the triangular Borel subalgebra, that is in the Lie algebra generated by $X$ and $H$.
 
 In $L$, we have that $-\frac{y''}{y'^2}$ vanishes when we apply the induced $X$ and $H$. So the image of $-\frac{y''}{y'^2}$ in $L/I$ belongs to the kernel of the action of $\mathfrak b$: namely, the algebraic closure of $\mathbb C [y,\frac{1}{P(y)}]$ in $L/I$. Let $z$ be this algebraic function.
 Direct computation shows that in $L/I$, $-\frac{y''}{y'^2}$ satisfies the following re-writing of equation $(\star')$
 
 $$
( -\frac{y''}{y'^2})' \frac{1}{y'}+ \frac{1}{2}(-\frac{y''}{y'^2})^2 + R(y) =0
 $$
 meaning that $\frac{z}{2}$ is an algebraic solution of
 $$
 \frac{d u}{dy} + u^2 +\frac{1}{2} R(y) =0
 $$
 in $\mathbb{C}(y)^{alg}$. This contradicts  Condition \ref{Ric}.
 \end{proof}
The next section is devoted to proving that Condition \ref{Ric} holds for Equation (\ref{stareqn}).

\section{The General Proof of Strong minimality} \label{Liouville}

\subsection{Liouvillian solutions, algebraic solutions, and Picard-Vessiot theory}


\begin{defn} Fix a differential field $K$ extending $\m C(y)$ such that the derivation on $K$ extends $\frac{d}{dy}$. We say that $K$ is \emph{Liouvillian} if there is a tower of field extensions $\m C(y) \subset K_0 \subset K_1 \subset \ldots \subset K_n = K$ such that for each $i=1, \ldots , n$, $K_i/K_{i-1}$ is generated by an element $a_i$ such that one of the following holds: 
\begin{enumerate}
\item $a_i ' \in K_{i-1}.$ 
\item $\frac{a_i'}{a_i} \in K_{i-1}.$ 
\item $a_i \in K_{i-1} ^{alg} .$ 
\end{enumerate} 
If $K$ is a field of meromorphic functions, then in case 1, $a_i= \int f$ for some $f \in K_i$ and in case 2, $a_i = e^{\int f} $ for some $f \in K_i$. So, occasionally we will refer to these cases as integrals or exponentials of integrals. 
\end{defn}

Consider the differential equation \begin{equation} \label{o2} z'' + p z' + q z = 0 \end{equation} where $p,q$ are rational functions in $\m C (y)$. The classification of its Liouvillian solutions has been extensively studied, and in \cite{Kovacic}, an algorithmic solution to determining the Liouvillian solutions was given. 

Let $z$ be a solution to equation \ref{o2}, and let $v= e^{\frac{1}{2} \int p } z.$ It follows by direct computation that 
\begin{equation} \label{Kov1} 
v'' + (b - \frac{1}{4} a^2 - \frac{1}{2}a' ) v =0
\end{equation}  
Because the previous transformation only involves scaling by a Liouvillian element, the Liouvillian solutions of equation \ref{Kov1} are in bijective correspondence with the Liouvillian solutions to equation \ref{o2}, and so without loss of generality, we may now assume that the order two equation in which we are interested is given in the following normal form: 
\begin{equation} \label{Kov2} 
z'' = r(y) z 
\end{equation} 
where $r(y) \in \m C(y)$.
 
\begin{thm}\label{KovLiou} \cite[page 5]{Kovacic} 
With regard to the Liouvillian solutions of a second order linear differential equation with coefficients in $\m C(y)$, there are four mutually exclusive options: 
\begin{enumerate} 
\item The differential equation \ref{Kov2} has a solution of the form $e^{ \int w}$ where $w \in \m C(y)$. 
\item The differential equation \ref{Kov2} has a solution of the form $e^{ \int w}$ where $w \in \m C(y)^{alg}$ is an algebraic function of degree two over $\m C(y)$. 
\item All of the solutions of \ref{Kov2} are algebraic over $\m C(y)$. 
\item No solution of \ref{Kov2} are Liouvillian. 
\end{enumerate} 
\end{thm} 

The connection with Riccati equations is as follows. If we define $u = \frac{z'}{z}$ where $z$ is a solution to equation \ref{Kov2}, then via direct computation we have that \begin{eqnarray} \label{ric1} u' + u^2 - r(y) = 0. \end{eqnarray} 
Notice that $z=ce^{\int u}$ for some constant $c\in \mathbb{C}$ and in particular $z_1=e^{\int u}$ is also a solution to \ref{Kov2}. So using Theorem \ref{KovLiou} we have the following lemma.
 
 \begin{lem} 
The Riccati equation \ref{ric1} has an algebraic solution over $\m C(y)$ if and only if the second order linear differential equation \ref{Kov2} has a Liouvillian solution. 
 \end{lem} 
 
Now, the verification of Condition \ref{Ric} follows from showing that equation \ref{Kov2} has no Liouvillian solutions. For this, we will need the following well-known result.

\begin{thm}\cite[page 8, case 4]{Kovacic}\label{Sl2} 
Let $G$ be the Picard-Vessiot group of \ref{Kov2}. There are no Liouvillian solutions to equation \ref{Kov2} if and only if $G = SL_2(\m C)$. 
\end{thm}

In the next subsection we will prove that, in the special case of a Fuchsian group $\Gamma$, the Picard-Vesiot group of the order two linear equation associated to the Riccati equation ($\star\star$) is $SL_2(\m C)$.

\subsection{Monodromy and the PV-group}

At this point, let us recall that the the Schwarzian equation $(\star)$ we focus on is given with
\[R_{j_{\Gamma}}(y)=\frac{1}{2}\sum_{i=1}^{r}{\frac{1-\alpha_i^2}{(y-a_i)^2}}+\sum_{i=1}^{r}{\frac{A_i}{y-a_i}},\] where the $\alpha_i$'s, $A_i$'s and $a_i$'s are obtained from the fundamental domain for $\Gamma$-action on $\mathbb{H}$.
As discussed in the previous subsection, if the Riccati equation corresponding to (\ref{stareqn})
\begin{equation}\label{RiccatiJ}
\frac{du}{dy}+u^2+\frac{1}{2}R_{j_{\Gamma}}(y) =0.
\end{equation}
were to have an algebraic solution $f\in\mathbb{C}(y)^{alg}$, then as in the previous subsection $z=e^{\int f}$ is a Liouvillian solution of the linear equation
\begin{equation}\label{fuchsianstar}
\frac{d^2z}{dy^2}+\left(\frac{1}{4}\sum_{i=1}^{r}{\frac{1-\alpha_i^2}{(y-a_i)^2}}+\sum_{i=1}^{r}{\frac{A_i/2}{y-a_i}}\right)z =0.
\end{equation}
This equation is an example of the most general (normal) form of a Fuchsian equation of second order:

\begin{defn} Consider the linear equation $\frac{d^2z}{dy^2} + a_1 \frac{dz}{dy} + a_2 z = 0$, where $a_1,a_2$ are rational functions in $\m C (y)$.
\begin{enumerate}
    \item A point $p\in\mathbb{C}$ is called {\em regular} if the functions $a_i$ have no pole at $p$, otherwise $p$ is called {\em singular}. To determine whether the point $y=\infty$ is regular, one simply substitutes $y=z^{-1}$ in the equation and verifies whether $z=0$ is regular for the new equation.
    \item  A point $p\in\mathbb{C}$ (resp. $p=\infty$) is called {\em regular singular} if it is singular and for each $i=1,2$, the limit $\text{lim}_{y\rightarrow p}(y-p)^ia_i(y)$ (resp. $\text{lim}_{y\rightarrow \infty}y^ia_i(y)$) exists and is finite.
    \item The equation is called {\em Fuchsian} if all points of ${\bf P}^1(\mathbb{C})$ are regular or regular singular.
\end{enumerate}
\end{defn}
It is well-known (cf. \cite[Chapter 7]{CresHaj}) that if the equation is Fuchsian then the coefficients $a_1$ and $a_2$ are of the form
\[a_i(y)=\frac{B_i(y)}{\prod_{j=1}^s (z-\beta_j)^i}\]
where $B_i(y)$ is a polynomial of degree $\leq i(s-1)$.

\begin{rem}
We have already seen in the previous section how to obtain the normal form of the a second order linear equation (see equation \ref{Kov1})
\end{rem}

As it turns out, the problem of existence of Liouvillian solutions for Fuchsian equations of second order is a classical one. We direct the reader to \cite{Gray} and \cite{HPSG} for some historical perspectives. We will only review parts of the theory that is relevant to this paper. Our focus will be the work of Poincar\'e on the relationship between the {\it monodromy group} of the Fuchsian equation \ref{fuchsianstar} and `its' Fuchsian group $\Gamma$. It is this work - partly rediscovering Schwarz's uniformization of ${\bf P}^1(\mathbb{C})$ by the $j_{\Gamma}$'s -  that lead Poincar\'e to introduce the theory of Fuchsian groups and functions, and to attack the problem of the uniformization of other Riemman surfaces.

\par From now on, we assume that the equation 
\begin{equation}\label{fuchsian}
\frac{d^2z}{dy^2} = r(y) z
\end{equation} 
is Fuchsian and denote by $S$ its set of singular points. For $z\in{\bf P}^1(\mathbb{C})\setminus S$, let $f_1$ and $f_2$ be analytic solutions in a neighborhood of $z$. We also assume that $f_1$ and $f_2$ are a basis of solutions, {\it i.e.,} that they are linearly independent over $\mathbb{C}$. Given any $\gamma\in\pi_1({\bf P}^1(\mathbb{C})\setminus S;z)$, we can analytically continue $f_1$ and $f_2$ along $\gamma$ and obtain new solutions $\tilde{f}_1$ and $\tilde{f}_2$ of \ref{fuchsian}. So there exists a matrix $M_{\gamma}\in GL_2(\mathbb{C})$ such that 
 \[
 \begin{pmatrix}
 \tilde{f}_1 \\           
 \tilde{f}_2
 \end{pmatrix} = M_{\gamma}\cdot
 \begin{pmatrix}
 f_1 \\
 f_2
 \end{pmatrix}
 \]
The mapping $\rho:\pi_1({\bf P}^1(\mathbb{C})\setminus S;z)\rightarrow GL_2(\mathbb{C})$, taking $\gamma\mapsto M_{\gamma}$ is a group homomorphism called the monodromy representation. Its image $M$ is called the monodromy group of equation \ref{fuchsian}. From the monodromy group, one can determine the Picard-Vessiot group of the equation: 

\begin{fct}\label{galois} \cite[Chapter 7]{CresHaj}
Let $G$ be the Picard-Vessiot group of the Fuchsian equation \ref{fuchsian}.  Then,
\begin{enumerate}
\item $G\subseteq{SL_2(\mathbb{C})}$;
\item if $M$ is its monodromy group, then $G$ is the Zariski closure of $M$.
\end{enumerate}
\end{fct}

Note that in particular from (1), for the Fuchsian equation \ref{fuchsianstar}, the monodromy group $M$ is a subgroup of ${SL_2(\mathbb{C})}$. We will now explain how in the case of equation \ref{fuchsianstar}, the monodromy group $M$ is related to the Schwarzian equation. The following well-known fact - which can be easily verified - will be needed.

\begin{fct}
Let $t(y)=j^{-1}_{\Gamma}(y)$ be a branch of the inverse of $y=j_{\Gamma}(t)$. Then $t(y)$ satisfies the following equation 
\begin{equation}\label{InverseODE}
S_{\frac{d}{dy}}(t)=R_{j_{\Gamma}}(y).
\end{equation}
Furthermore, the functions
\[z_1=\frac{t}{(\frac{dt}{dy})^{\frac{1}{2}}}\;\;\;\;z_2=\frac{1}{(\frac{dt}{dy})^{\frac{1}{2}}}\]
form a basis of solutions of the Fuchsian equation \ref{fuchsianstar}
\[\frac{d^2z}{dy^2}+\frac{1}{2}R_{j_{\Gamma}}(y)z =0.\] 
\end{fct}

Notice that in particular $\frac{z_1(y)}{z_2(y)}=t(y)$. This allows us to define from $M$ the projective monodromy of the equation \ref{InverseODE}. Namely, if $M_{\gamma}=\begin{pmatrix}
    a & b \\
    c & d
  \end{pmatrix}\in SL_2({\mathbb{C}})$ is monodromy matrix (as above), then
  \[t^*=\frac{at+b}{ct+d}\]
  is again a solution of the equation \ref{InverseODE}. The collection of matrices $\hat{M}_{\gamma}:t\mapsto t^*$ is called the projective monodromy group $\hat{M}$ of the equation \ref{InverseODE}. Of course $\hat{M}$ is the image of $M$ under the natural projection $\pi:SL_2(\mathbb{R})\rightarrow PSL_2(\mathbb{R})$. The following proposition is attributed to Poincar\'e in various sources but we know of no reference for a proof of it and thus reproduce it here.
 
 \begin{prop}
 The projective monodromy group of the equation \ref{InverseODE} is $\Gamma$. As a consequence, the monodromy group of the Fuchsian equation \ref{fuchsianstar} is $\pi^{-1}(\Gamma)$.
 \end{prop}
 \begin{proof} 
 Throughout $t(y)=j^{-1}_{\Gamma}(y)$ is a branch of the inverse of $j_{\Gamma}$ locally defined on some small domain $U$ and $\hat{M}$ is the projective monodromy group.
\par We have 
\begin{eqnarray*}
g\in\hat{M}\setminus \{I\} &\iff& gt(y) \text{ is another branch of the inverse of } j_{\Gamma} \text{ (defined on some}\\ 
& &\text{ larger domain $U'$).}\\
 &\iff& y=j_{\Gamma}(t(y))=j_{\Gamma}(gt(y))\\
 &\iff& g\in\Gamma\setminus \{I\}.
\end{eqnarray*}
We have used here that $j_{\Gamma}$ is a globally defined single-valued function. 
 \end{proof}

\begin{prop}
There are no Liouvillian solutions of the Fuchsian equation \ref{fuchsianstar}. Consequently, Condition \ref{Ric} holds for the Riccati equation \ref{RiccatiJ}. 
\end{prop}
\begin{proof}
We have that $\pi^{-1}(\Gamma)$, the monodromy group of the Fuchsian equation \ref{fuchsianstar}, is Zariski dense in $SL_2(\mathbb{C})$. Hence by Fact \ref{galois}, the Picard-Vessiot group is $G=SL_2(\mathbb{C})$. By Theorem \ref{Sl2} there are no Liouvillian solutions for the equation.
\end{proof}
 
We thus obtain the first part of Theorem \ref{Geotrivial}; namely the Schwarzian equation (\ref{stareqn}) is strongly minimal.

\section{Geometric triviality and Algebraic relations} \label{algreltriv}
\subsection{The classification of strongly minimal sets}\label{trivial}

In this section we will discuss some general model-theoretic results regarding strongly minimal sets in differentially closed fields. In particular, we will explain some consequences of the (unpublished\footnote{A complete proof, can be found in \cite[Corollary 3.10]{PilZie} and in the arguments in the paragraphs leading up to Proposition 4.10 of \cite{Pil}. A good guide/summary of the proof can also be found in \cite[Section 2.1]{NagPil}.}) work of Hrushovski and Sokolovi\'c on the classification of strongly minimal sets. We will, from these considerations, obtain geometric triviality of the Schwarzian equations satisfied by the uniformizing functions in the earlier sections. Let us denote by $\mathcal{U}$ the differentially closed field of characteristic zero that we work in. We assume that $\mathcal{U}$ is saturated and that $\mathbb{C}$ (defined by $y' = 0$) is its field of constants. Notice incidentally that $\mathbb{C}$ is itself a strongly minimal definable set. Indeed, it is the only definable strongly minimal subfield of $\mathcal{U}$. In what follows strongly minimal sets are understood to be defined over some finitely generated differential subfield $K$ of $\mathcal{U}$.

\par The zero set of any irreducible order one differential polynomial in a single variable (by irreducible, we will always mean as a polynomial) is also strongly minimal. Higher order linear differential equations are never strongly minimal (one can define linear subspaces using elements of a fundamental set of solutions). For higher order non-linear equations, it seems that it is in general difficult to establish strong minimality. However, if the strong minimality of an equation is established, one can often employ a variety of model theoretic tools to establish even stronger results. 

Other important examples of strongly minimal sets are given by the following
\begin{fct}[\cite{Bui},\cite{HruSok}] Let $A$ be an abelian variety defined over ${\mathcal U}$. We identify $A$ with its set $A({\mathcal U})$ of ${\mathcal U}$-points. Then
\begin{enumerate}
\item $A$ has a (unique) smallest Zariski-dense definable subgroup, which we denote by $A^{\sharp}$. 
\item If $A$ is a simple abelian variety that does not descend to $\mathbb{C}$, then $A^{\sharp}$ is strongly minimal.
\end{enumerate}
\end{fct}

The subgroup $A^{\sharp}$ is called the \emph{Manin kernel} of $A$ (cf. \cite{Mark}). The trichotomy theorem, gives a classification of strongly minimal sets up to \emph{non-orthogonality}, a notion we will explain following the statement of the theorem. 

\begin{Thm}[\cite{HruSok},\cite{PilZie}]\label{trichotomy} Let $Y$ be a strongly minimal set. Then exactly one of the following holds:
\begin{enumerate}
\item $Y$ is nonorthogonal to the strongly minimal set $\mathbb{C}$,
\item $Y$ is nonorthogonal to $A^{\sharp}$ for some simple abelian variety $A$ over ${\mathcal U}$ which does not descend to $\mathbb{C}$, 
\item $Y$ is geometrically trivial.
\end{enumerate}
\end{Thm}

\begin{defn} 
Let $Y$ and $Z$ be strongly minimal sets. Denote by $\pi_1:Y\times Z\rightarrow Y$ and $\pi_2:Y\times Z\rightarrow Z$ the projections to $Y$ and $Z$ respectively. We say that $Y$ and $Z$ are \emph{nonorthogonal} if there is some infinite definable relation $R\subset Y\times Z$ such that ${\pi_1}_{|R}$ and ${\pi_2}_{|R}$ are finite-to-one functions.
\end{defn} 

The sets $Y$ and $Z$ are defined over some finitely generated differential subfield $K$ of $\mc U$, and so for any differential field $F$ containing $K$, it makes sense to ask whether a relation $R$ as above can be defined over $F$.

\begin{defn} We say that $Y$ is \emph{weakly orthogonal} to $Z$ over $F$ if no such relation $R$ can be defined over $F$. 
\end{defn} 
The following facts from the model theory of differential fields are well-known (see for instance, \cite{MarkI}).

\begin{fct} Let $Y$ and $Z$ be strongly minimal sets.
\begin{enumerate} 
\item Nonorthogonality is an equivalence relation on strongly minimal sets. 

\item Nonorthogonality classes of strongly minimal differential equations refine various basic invariants of the equations. For instance, if $Y,Z$ are nonorthogonal then ${\rm order}(Y) = {\rm order}(Z)$. Recall that the order of a definable set $Y$ is given by ${\rm order}(Y)=\sup\{{\rm tr.deg.}_KK\gen{y}:y\in Y\}$, where $K$ is any countable differential field over which $Y$ is defined.








\item If $Y$ and $Z$ are nonorthogonal, then they fall into the same category of Theorem \ref{trichotomy}. 

\item Strongly minimal sets that fall in cases (2) and (3) of Theorem \ref{trichotomy} are said to be \emph{locally modular} (and in case (1) the sets are \emph{non-locally modular}).

\item Orthogonality has a natural interpretation in terms of transcendence. Suppose that $Y$ and $Z$ are orthogonal strongly minimal sets defined over $K$. Let $a,b$ be solutions of $Y, Z$, respectively. Let $F$ be any differential field extending $K$. Then 
$$ {\rm tr.deg.}_F(F\langle a,b \rangle) = {\rm tr.deg.}_F(F\langle a \rangle) +{\rm tr.deg.}_F(F\langle b \rangle).$$ Conversely, if the inequality does not hold for some $a \in Y$ and $b \in Z$ over $F$, then $Y$ is not weakly orthogonal to $Z$ over $F$. 
\end{enumerate}
\end{fct}

Nonorthogonality of Manin Kernels has been further classified in terms of isogeny classes of abelian varieties.
\begin{fct}\label{isogeny} If $A$ and $B$ are two simple abelian varieties which do not descend to $\mathbb{C}$, then $A^{\sharp}$ and $B^{\sharp}$  are non-orthogonal if and only if $A$ and $B$ are isogenous.
\end{fct}

For relations $R$ which witness nonorthogonality between trivial strongly minimal sets, there is an important and very general descent result:  
\begin{fact} \label{nonewparam} \cite[Corollary 2.5.5]{GST} Two geometrically trivial strongly minimal sets are nonorthogonal if and only if they are non weakly orthogonal. That is, the relation $R$ witnessing nonorthogonality of $X$ and $Y$ can be defined over the differential field generated by the parameters used in the equations defining $X$ and $Y$. 
\end{fact} 

\begin{prop}
Let $Y$ be a strongly minimal set of order $>1$ and suppose that $Y$ is defined over $\m C$. Then $Y$ is geometrically trivial.
\end{prop}

\begin{proof}\footnote{We thank Dave Marker for a sketch of this proof.}
First note that since ${\rm order}(Y)\neq 1$, $Y$ is necessarily orthogonal to the constants ${\m C}$. So by Theorem \ref{trichotomy}, to show that $Y$ is geometrically trivial, we only need to show that it is orthogonal to all Manin kernels. We argue by contradiction.
\par Suppose that $Y$ is nonorthogonal to $A^{\sharp}$ for some simple abelian variety $A$ over ${\mathcal U}$ which does not descend to $\mathbb{C}$. Let $(A,\lambda)$ be a principal polarization of $A$. We can use the fact that moduli spaces of principally polarized abelian varieties exist over any base field (cf. \cite[Chapter 7]{Mum}). So let $(V,\varphi)$ be a moduli space for $(A,\lambda)$ over $\m C$. For some $b$ in $(V,\varphi)$, we have that $(A,\lambda)=(V_b,\varphi_b)$.
\par Using uniform definability of Manin kernels \cite[Lemma 2.25]{NagPil}, we have a formula $\phi(x)$ over $\m C$
asserting that $Y$ is non-orthogonal to $V^{\sharp}_x$ and such that $\phi(b)$ is true in $\mathcal{U}$.  If $\phi(c)$ holds, 
then, by Fact \ref{isogeny}, it must be the case that $V_b$ and $V_c$ are isogenous. But there are only countably many abelian varieties isogenous to $V_b$.  Hence the definable set $\{a\in\mathcal{U}:\phi(a) \text{ is true in } \mathcal{U}\}$ is countable and so must be finite. In other words $c$ is algebraic over (and so in) $\m C$. But this contradicts the assumption that $A$ does not descend to $\mathbb{C}$.
\end{proof}

\begin{cor} \label{trivHS} For $\Gamma$ a Fuchsian group, equation (\ref{stareqn}) defines a geometrically trivial strongly minimal set. 
\end{cor} 
We have hence established the entirety of Theorem \ref{Geotrivial}.

\subsection{Transcendence and orbits of the commensurator of $\Gamma$}



\begin{Thm} \label{2relations}
Let $(K,\partial)$ be a differential extension of $(\mathbb C(t), \frac{\partial}{\partial t})$ with no new constants. Let $\Gamma$ be a Fuchsian group and $j_1$, $j_2$ be two solutions of the equation 
$$
S_{\frac{d}{dt}}(y) +(y')^2\cdot R_{\Gamma}(y) =0.
$$
If \[\td_{K} K(j_1, j_1', j_1'',j_2,j_2',j_2'')<6\] then $j_1$ or $j_2$ is algebraic over $K$ or there is a nonzero polynomial $P(y_1,y_2)$ over $\m C$ such that $P(j_1, j_2)=0$. 
\end{Thm}

The group $PSL_2(\m C)$ acts on pairs of solutions by precomposition. We will prove that the ideal of differential relations between $(j_1, j_2)$ is stable under this action.


\begin{proof} 
From Theorem \ref{juststrmin}, it follows that if
$\td _{K} K (j_1, j_1', j_1'' , j_2, j_2', j_2'') <6$ then it is either $0$ or $3$. But it follows from Fact \ref{nishioka1} and Lemma \ref{fct1} that if both $j_1$ and $j_2$ are not algebraic over $K$ then $\td_{K} K(j_1, j_1', j_1'',j_2,j_2',j_2'')=3.$ By Fact \ref{nonewparam}, we can assume that $K=\mathbb C(t)$ and so throughout $K=\mathbb C(t)$ and $\partial=\frac{\partial}{\partial t}$.

We proceed as in the proof of Theorem \ref{juststrmin}. For $i=1,2$ consider $K(y_i, y_i',y_i'')$, equipped with the derivation
\begin{itemize}
\item $D_i =  \partial + y_i' \frac{\partial}{\partial y_i} + y_i'' \frac{\partial}{\partial y'_i} +  \left( \frac{3}{2}\frac{y_i''^2}{y_i'} -(y_i')^3 R_{\Gamma}(y_i) \right) \frac{\partial}{\partial y''_i}$
\end{itemize}

One defines an action of $\mathfrak{psl}_2(\m C)$ by : 

\begin{itemize}
\item $ X_i = \partial$
\item $ H_i = t\partial - y_i'\frac{\partial}{\partial y_i'} -2 y_i'' \frac{\partial}{\partial y_i''}$
\item $ Y_i = \frac{t^2}{2}\partial -ty_i'\frac{\partial}{\partial y_i'} - (2ty_i'' +y_i')\frac{\partial}{\partial y_i''}$
\end{itemize}

We have that $[X_i, H_i] =  X_i$, $[H_i, Y_i]= Y_i$, $[X_i,  Y_i] = H_i$, and
$[D_i, X_i] = 0$, $[D_i,  H_i] = D_i$, $[D_i, Y_i] = t D_i$.

As explained in the proof of Theorem \ref{juststrmin}, the above action of $\mathfrak{psl}_2(\mathbb C)$ is the infinitesimal action of $PSL_2(\m C)$ on $\mathbb C(t,y_i,y'_i, y''_i)$. We ``verticalize'' this action by considering $X^v_i = X_i - D_i$,  $H^v_i = H_i - tD_i$ and $Y^v_i = Y_i - \frac{t^2}{2}D_i$. Now $\CC X^v_i + \CC H^v_i + \CC Y^v_i$ is a realization of $\mathfrak{psl}_2(\CC)$ acting $K$-linearly and commuting with $D_i$.


The ideal of the polynomial differential relations between $j_1$ and $j_2$ over $K$ is an ideal in $K[y_1, y_1', \ldots, y_2, y'_2, \ldots]$. Let $J$ be the differential ideal generated by the third order differential equations satisfied by $j_1$ and by $j_2$ and $K[y_1, y_1', \ldots, y_2, y'_2, \ldots] \to K(y_1,y'_1,y''_1) \otimes _{K} K(y_2,y'_2,y''_2)$ be the quotient by $J$ followed by localizations. 

As $j_1$ and $j_2$ do not satisfy any lower order differential equations this ideal is the preimage of an ideal $I$ of $L = K(y_1,y'_1,y''_1) \otimes _{K} K(y_2,y'_2,y''_2)$ stable by 
 \begin{multline}
 D^{(2)} = \partial + y_1' \frac{\partial}{\partial y_1} + y_2' \frac{\partial}{\partial y_2} + y_1'' \frac{\partial}{\partial y'_1} + y_2'' \frac{\partial}{\partial y'_2} \\ 
 + \left( \frac{3}{2}\frac{y_1''^2}{y_1'} - (y_1')^3  R_{\Gamma}(y_1)\right) \frac{\partial}{\partial y''_1} + \left( \frac{3}{2}\frac{y_2''^2}{y_2'} - (y_2')^3 R_{\Gamma}(y_2) \right)\frac{\partial}{\partial y''_2}.
 \end{multline}
    
    The ideal $I$ is the kernel of the evaluation in $(j_1, j_2)$ with values in a field of meromorphic function thus it is prime. From geometric triviality, the subfield of constants of $F = \text{Frac}(L/I)$ with respect to the derivation $D^{(2)}$ is $\mathbb C$.

    \medskip
    
{\it  We claim that $I$ is stable under the diagonal action of $\mathfrak{psl}_2$.  } 
 \medskip
 
The algebra $L/I$ is an algebraic extension of $K(y_1,y'_1,y''_1)$ and of $K(y_2,y'_2,y''_2)$, and as usual, $D_1$, $X_1$, $H_1$, $Y_1$, $D_2$, $X_2$, $H_2$, $Y_2$ and their "verticalization" will also denote their unique extensions to $L/I$.

\begin{lem}
On $L/I$ we have $D_1 = D^{(2)} = D_2$. 
\end{lem}
\begin{proof}
Restrict the derivation $D^{(2)}$ of $L/I$ to its subalgebra $L_1$. The definition of $D^{(2)}$ gives that this restriction is $D_1$. Now, the extension to the algebraic extension $L/I$ of $L_1$ is unique then $D^{(2)} = D_1$. Same argument gives $D^{(2)} = D_2$.
\end{proof}

{\it So we will just write this derivation as $D$. }

\begin{lem}\label{lem6}
There exists $a\in \CC$ such that, on $L/I$, $X_1 = X_2$, $H_1 = H_2 + a(X_2-D)$ and $Y_1= Y_2 +a (H_2 -tD) + \frac{a^2}{2}(X_2-D)$.
\end{lem}
 
 \begin{proof}
 Using Fact \ref{nonewparam}, we have that any algebraic relations between $j_1$ and $j_2$ (together with derivatives) can be defined over $\m C$. Hence $I$ is generated by $I \cap \mathbb C (y_1, y_1',y_1'') \otimes \mathbb C (y_2, y_2',y_2'')$. Then, on $L/I$, both $X_1$ and $X_2$ coincide with $\frac{\partial}{\partial t}$, this proves $X_1 = X_2$.
 
The two triples $X_1^v$, $H^v_1$, $Y^v_1$ and $X_2^v$, $H^v_2$, $Y^v_2$ are two bases of the derivations of $F=\text{Frac}(L/I)$ over $K$. Let $A$ be the matrix with coefficient in $F$ such that $(X_1^v,H^v_1,Y^v_1) =(X_2^v,H^v_2,Y^v_2)A$. From the bracket with $D$ one gets $0 = [D,(X_1^v,H^v_1,Y^v_1) ] = (X_2^v,H^v_2,Y^v_2)D(A)$. So the coefficients of $A$ are constant. 

Now the two triples are basis of two realisations of $\mathfrak{psl}_2 (\mathbb C)$ with the same structure constants, then $A$ is an automorphism of the Lie algebra $\mathfrak{psl}_2 (\mathbb C)$. All automorphisms of $\mathfrak{psl}_2 (\mathbb C)$ are inner (see \cite[Proposition 14.21]{SagleWalde}), thus there exists a $g\in PSL_2(\mathbb C)$ such that ${\rm Ad}(g) = A$. This automorphism fixes $X_1$, hence there exists $a\in \mathbb C$ such that $g= \begin{pmatrix} 1 & a \\ 0 & 1 \end{pmatrix}$. Then
 \[
 (X_1^v,H^v_1,Y^v_1) =(X_2^v,H^v_2,Y^v_2) \begin{pmatrix} 1 & a & \frac{a^2}{2} \\ 0 & 1 & a \\ 0 & 0 & 1\end{pmatrix}
 \]
 
This proves the lemma.
 \end{proof}
 
In $F$, $y_2$ is an algebraic function over $\mathbb C(t,y_1,y'_1,y''_1)$ satisfying $X_2(y_2) = H_2(y_2) = Y_2(y_2) = 0$. Hence, using Lemma \ref{lem6}, one easily computes that 
\begin{equation}\label{X} 
X_1(y_2)=0,
\end{equation}
\begin{equation}\label{H}
H_1(y_2) = -a D(y_2),
\end{equation}
\begin{equation}\label{Y}
Y_1(y_2) = (-at -\frac{a^2}{2} )D(y_2) = (t+\frac{a}{2})H_1(y_2). 
\end{equation}

We will prove that this system of partial differential equations over $\mathbb C(t,y_1,y'_1,y''_1)$ has an algebraic solution if and only if $a=0$. For contradiction, assume not.
We expand $y_2$ as a Puiseux series in $1/z$ with $z = \frac{y_1''}{y_1'^2} $, that is we think of $y_2$ as being an element of $\mathbb C(t,y_1,y'_1)^{alg}\gen{\gen{\frac{1}{z}}}$:
\[
y_2 = \sum_{\lambda\leq n} A_\lambda (t,y_1,y_1') z^\lambda.
\] 

In the coordinates $t,y_1, y_1', z$, one has
\begin{itemize}
\item $ X_1 = \frac{\partial}{\partial t}$,
\item $H_1 = t\frac{\partial}{\partial t} -    y'_1\frac{\partial}{\partial y'_1}$,
\item $Y_1 = \frac{t^2}{2}\frac{\partial}{\partial t} - ty'_1\frac{\partial}{\partial y'_1} -\frac{1}{y_1'}\frac{\partial}{\partial z}$,
\item $D = \frac{\partial}{\partial t} - y'_1\frac{\partial}{\partial y_1} + z(y_1')^2\frac{\partial}{\partial y_1'} -\left(\frac{1}{2}z^2y_1' + R_\Gamma(y_1)y_1' \right)\frac{\partial}{\partial z}$.
\end{itemize}
The induced continuous action of $X_1$ on $\mathbb C(t,y_1,y'_1)^{alg}\gen{\gen{\frac{1}{z}}}$
gives

\begin{equation}\label{puiseuxX}
X_1(y_2)=\sum_{\lambda \leq n} \frac{\partial A_\lambda}{\partial t} z^\lambda,
\end{equation}
The equations \ref{X} and \ref{puiseuxX} give for all $\lambda$, $\frac{\partial A_\lambda}{\partial t} = 0$.
Then by direct computation one gets
\begin{equation}\label{puiseuxH}
H_1(y_2) = \sum_{\lambda \leq n} - y_1'\frac{\partial A_\lambda}{\partial y_1'} z^\lambda,
\end{equation}
\begin{equation} \label{puiseuxY}
Y_1(y_2) = \sum_{\lambda \leq n} -ty'_1\frac{\partial A_\lambda}{\partial y'_1}(z)^\lambda  - \lambda A_\lambda \frac{1}{y'_1}z^{\lambda-1},
\end{equation}
\begin{equation}\label{puiseuxD}
D(y_2) = \sum_{\lambda \leq n} R_\Gamma(y_1) y'_1 \lambda A_\lambda z^{\lambda-1}+ y'_1\frac{\partial A_\lambda}{\partial y_1} z^\lambda + \left( (y'_1)^2\frac{\partial A_\lambda}{\partial y'_1} - \frac{1}{2}y'_1 \lambda A_\lambda \right)z^{\lambda +1}.
\end{equation}

\begin{lem}\label{Hy}
If $y_2$ is an algebraic solution of \ref{X}, \ref{H}, \ref{Y} then $H_1(y_2)=0$.
\end{lem}

\begin{proof}
If $a=0$ there is nothing to prove. Assume it is not. 
We have already seen that $\frac{\partial A_\lambda}{\partial t}= 0$. 
Let $q\in \mathbb Q$ be such that $A_q \not = 0$. One can assume that $q$ is maximal among the elements $q' \in q + \mathbb Z$ such that $A_{q'} \not = 0$. From \ref{Y}, one gets $-ty'_1 \frac{\partial A_q}{\partial y'_1} = (t+\frac{a}{2})\left( -y'_1 \frac{\partial A_q}{\partial y'_1}\right)$ and then $\frac{\partial A_q}{\partial y'_1} = 0$. Now \ref{H} gives  $(y'_1)^2\frac{\partial A_q}{\partial y'_1} - \frac{1}{2}y'_1 q A_q = 0$, this implies $q=0$ so that $n=0$ and the range of $\lambda$ is $-\mathbb N$.

The equation \ref{Y} can be written as: $\forall\ k \in \mathbb N$,
\[
\tag{\ref{Y} ($k$)}\frac{a}{2} \frac{\partial A_{-k-1}}{\partial y'_1} = k \frac{A_{-k}}{y'_1}
\]

Let $k_0$ be the maximal integer such that for all strictly positive $k$ smaller than $k_0$, $A_{-k}=0$.
The equality \ref{Y}\,($0$) gives that $A_{-1}$ does not depend on $y'_1$. Then \ref{Y}\,($1$) is an equality between a derivative of an algebraic function in $y_1'$ and and rational function with a simple pole at $0$. This implies that the latter is identically zero: $k_0$ is greater than $2$.

Now if $k_0$ is finite then \ref{Y}\,($k_0-1$) is $\frac{\partial A_{-k_0}}{\partial y'_1} = 0$ and \ref{Y}\,($k_0$) is $\frac{\partial A_{-k_0-1}}{\partial y'_1} =\frac{A_{-k_0}}{y'_1}$.
As a derivative of an algebraic function can not have simple pole, $A_{-k_0}=0$ which contradicts the existence of $k_0$.

Then \ref{puiseuxH} proves the lemma.
\end{proof}

If $a\not = 0$ Lemma \ref{Hy} and the equation \ref{H} show that $D(y_2) = 0$. But the subfield of constants of $D$ in $F$ is $\mathbb C$ and $y_2$ is not contant. This contradicts the assumption on $a$ and one gets $a=0$.

Now, on $F$, $X_1=X_2$, $H_1=H_2$ and $Y_1 = Y_2$. These three derivations are linearly independent and their kernel is denoted by $N$. Formulas for these derivations give $y_1 \in N$ and $y_2 \in N$.

The sequence of extensions
$
\mathbb C \subset N \subset F
$ is such that $\td_{\mathbb C} N \geq 1$, $ \td_N F\geq 3 $ and $\td_\CC  F =4 $ then the transcendence degree of $N$ over $\mathbb C$ is $1$. This proves that $I$ contains some nonzero $P \in \mathbb C[y_1, y_2]$. It is not difficult to see that $P$ generates $I$ as a $D$-ideal.
\end{proof}
\begin{rem}
It is not hard to see that Theorem \ref{2relations} also holds for all general Schwarzian equations $(\star')$ provided that they are strongly minimal (and so geometrically trivial). Indeed the above proof did not use the fact that Fuchsian groups are involved. In particular, Theorem \ref{2relations} holds if Condition \ref{Ric} is true of the corresponding Riccati equations. 
\end{rem}

It now remains to understand the kind of polynomials $P \in \mathbb C[y_1, y_2]$ that can occur. Notice that if $P(j_\Gamma(g_1t),j_\Gamma(g_2t))=0$ gives an algebraic relation between two solutions $j_\Gamma(g_1t)$ and $j_\Gamma(g_2t)$, then there trivially is an algebraic relation between $j_\Gamma(t)$ and $j_\Gamma(g_2g_1^{-1}t)$, namely $P(j_\Gamma(t) , j_\Gamma(g_2g_1^{-1}t) )=0$. So it suffices to characterize interalgebraicity with $j_\Gamma(t)$.
 
 

 \begin{lem} \label{nonalg} For $g_1 \notin {\rm Comm}( \Gamma ) ,$ $j_\Gamma(t) $ is algebraically independent from $j_{\Gamma } (gt)$ over $\m C$. 
 \end{lem} 
 \begin{proof}  Let $g \notin \text{Comm} (\Gamma).$ For a contradiction, assume first that $P$ is an algebraic relation over $\m C$ holding between $j_\Gamma(t) $ and $j_{\Gamma } (gt)$. 
Then for all $a \in \m H$, we have that $P(j_\Gamma (a), j_ \Gamma (ga)) = 0.$ For $\gamma \in \Gamma$, consider the point $b_ \gamma = \gamma \cdot a$. Letting $a= b_ \gamma$, we have that $P(j_\Gamma (b_\gamma), j_ \Gamma (gb_\gamma )) = 0.$ 
 
 But, since $j_\Gamma (b_\gamma) = j_\Gamma ( a)$, we have that $P(j_\Gamma (a), j_ \Gamma (gb_\gamma )) = 0.$ Now, by the $\Gamma $-invariance of $j_\Gamma$, we have that for any $\gamma_1 \in \Gamma$,  $P(j_\Gamma (\gamma_1 a), j_ \Gamma (\gamma_1 g \gamma a )) =0$. But $j_\Gamma (\gamma_1 a ) = j_\Gamma (a)$, we have that 
 $$
 P(j_\Gamma (a) , j_\Gamma (\gamma_1 g \gamma a)) =0
 $$ f
 or all $ \gamma_1, \gamma \in \Gamma.$ However, $j_\Gamma$ is precisely $\Gamma$-invariant, and for $g \notin \text{Comm}(\Gamma)$, there are infinitely many left coset representatives of $\Gamma$ among the double coset $\Gamma g \Gamma$. Then there are infinitely many distinct points $y$ for which $P(j_\Gamma (a ) , y) =0$ holds, contradicting the fact that $P=0$ gives an algebraic relation. 
  
 \end{proof} 
 
 \begin{lem}\cite[Section 7.2]{Shim} \label{isalg} For $g \in {\rm Comm}( \Gamma ) ,$ $j_\Gamma(t) $ is algebraically dependent with $j_{\Gamma } (gt)$ over $\m C$.  
 \end{lem} 
 
\begin{defn} By Lemma \ref{isalg}, when $g \in \text{Comm}( \Gamma ) ,$ there is an irreducible polynomial $\Psi_{\tilde{g}}(x,y) \in \m C[x,y]$ such that $\Psi_{\tilde{g}}(j_\Gamma(t) , j_\Gamma(gt) )=0.$ We call $\Psi_{\tilde{g}}$ a \emph{$\Gamma$-special polynomial}, and the zero set of such a polynomial a \emph{$\Gamma$-special curve}. 
\end{defn} 

{
Now from Theorems \ref{Geotrivial} and \ref{2relations} and Lemmas \ref{nonalg} and \ref{isalg}, one gets the weak form of the Ax-Lindemann-Weierstrass Theorems \ref{ALW} and \ref{ALW bis}.

\begin{Thm}\label{ALW1}
Let $K$ be a differential extension of $(\CC(t), \frac{\partial}{\partial t})$  and $j_{\Gamma}(g_{1}t),...,j_{\Gamma}(g_{n}t)$ be distinct solutions of the Schwarzian equation $(\star)$ that are not algebraic over $K$ nor pairwise related by $\Gamma$-special polynomials. Then the $3n$ functions 
\[j_{\Gamma}(g_{1}t),j'_{\Gamma}(g_{1}t),j''_{\Gamma}(g_{1}t),\ldots,j_{\Gamma}(g_{n}t),j'_{\Gamma}(g_{n}t),j''_{\Gamma}(g_{n}t)\] 
are algebraically independent over $K$. 
\end{Thm}

\begin{proof}

For contradiction, assume that the $3n$ functions 
\[j_{\Gamma}(g_{1}t),j'_{\Gamma}(g_{1}t),j''_{\Gamma}(g_{1}t),\ldots,j_{\Gamma}(g_{n}t),j'_{\Gamma}(g_{n}t),j''_{\Gamma}(g_{n}t)\] 
are algebraically dependent over $K$. Define the field $\tilde K$ as 
\begin{eqnarray}\notag \tilde K&=&K\left(j_{\Gamma}(g_{2}t),j'_{\Gamma}(g_{2}t),j''_{\Gamma}(g_{2}t)\ldots,j_{\Gamma}(g_{n}t),j'_{\Gamma}(g_{n}t),j''_{\Gamma}(g_{n}t)\right)\\\notag
&=& K\gen{j_{\Gamma}(g_{2}t),\ldots,j_{\Gamma}(g_{n}t)}.
\end{eqnarray}
By strong minimality of equation (\ref{stareqn}), it must be that $j_{\Gamma}(g_1t)\in \tilde K^{alg}$ and by geometric triviality of (\ref{stareqn}), we have that
\[j_{\Gamma}(g_1t)\in K\gen{j_{\Gamma}(g_{i}t)}^{alg}\]
for some $i=2,\ldots,n$. Using Theorem \ref{2relations} we get that 
\[j_{\Gamma}(g_1t)\in{\mathbb C}(j_{\Gamma}(g_{i}t))^{alg}\]
and so
\[j_{\Gamma}(t)\in{\mathbb C}(j_{\Gamma}(g_{i}g^{-1}_{1}t))^{alg}\]
Now using  Lemma \ref{nonalg}, it must be the case that $g=g_{i}g^{-1}_{1}\in Comm(\Gamma)$. So for the $\Gamma$-special polynomial $\Psi_{\tilde{g}}$, we get
\[\Psi_{\tilde{g}}(j_{\Gamma}(t),j_{\Gamma}(g_{i}g^{-1}_{1}t))=0\]
and hence
\[\Psi_{\tilde{g}}(j_{\Gamma}(g_{1}t),j_{\Gamma}(g_{i}t))=0.\]
This contradicts our assumption that $j_{\Gamma}(g_{1}t)$ and $j_{\Gamma}(g_{i}t)$ are not related by any $\Gamma$-special polynomials.
\end{proof}

}

\section{Orthogonality and the Ax-Lindemann-Weierstrass theorem} \label{OrthALW}

In the previous sections, we have understood the structure of the solution set of 
$$S_{\frac{d}{dt}}(y) +(y')^2\cdot R_{j_{\Gamma}}(y) =0.$$ Define 
\begin{eqnarray} \label{chieqn} \chi _{\Gamma, \frac{d}{dt}} (y) = S_{\frac{d}{dt}}(y) +(y')^2\cdot R_{j_{\Gamma}}(y). 
\end{eqnarray}
In this section, we consider equations of the form $\chi _{\Gamma, \frac{d}{dt}}  (y) = a$ for $a$ an element in some differential field extension of $\m Q$, and produce a similar analysis. 

\subsection{Strong minimality and algebraic relations on other fibers} 
First, we prove the solution set of the equation $\chi _{\Gamma, \frac{d}{dt}}  (y) = a$ is strongly minimal and characterize the algebraic relations between solutions. Essentially, the analysis from \cite[Section 5.1]{FreSca} adapts to this case, but for the sake of completeness, we will provide a brief explanation here.

Let $a \in K$ be an element in some differential field extension of $\m Q$. By Seidenberg's embedding theorem, we can, without loss of generality, assume $a=a(t)$ is given by a meromorphic function over some domain $U$, and the derivation is given by $\frac{d}{dt}$. After sufficently shrinking the domain, there is some meromorphic function $\tilde a (t)$ satisfying $S_{\frac{d}{dt}}(\tilde a) = a$ such that $$\chi _{\Gamma, \frac{d}{dt}}  (j_{\Gamma}(\tilde a(t)) ) = a(t).$$ 

The following Lemma follows by the Schwarzian chain rule and is nearly identical to \cite[Lemma 5.1]{FreSca}:

\begin{lem} \label{varchangelem} 
Let $K$ be a differentially closed $\frac{d}{dt}$-field containing $a$. There exists $\partial \in K\frac{d}{dt}$ such that $\chi _{\Gamma, \partial}  (y) =0$.

\end{lem} 
\begin{proof}
The equation $S_{\frac{d}{dt}}(\tilde a) = a$, with unknown $\tilde a$, can be considered as a differential equation over $\m C\langle a \rangle$. By Seidenberg's theorem this field can be assumed to be a field of meromorphic functions on some domain $U\subset \m C$ and by the usual Cauchy theorem, one can build a solution, holomorphic on some domain $U' \subset U$.

By the differential Nullstellensatz there exists $\tilde a \in K $ a solution of $S_{\frac{d}{dt}}(\tilde a) = a$. Then $\partial = \frac{1}{{\tilde a}'}\frac{d}{dt}$.
\end{proof}

\begin{thm} \label{Strminfiber} The sets defined by $\chi _{\Gamma, \frac{d}{dt}} (y) = a$ are strongly minimal and geometrically trivial. If $a_1, \ldots , a_n$ satisfy $\chi _{\Gamma, \frac{d}{dt}} (a_i) = a$ and are dependent, then there exist $i,j \leq n$ and a $\Gamma$-special polynomial $P$ such that $P(a_i, a_j) = 0$.
\end{thm}


The proof of Theorem \ref{Strminfiber} is quite similar to that of \cite{FreSca} Proposition 5.2, but we include it here for completeness. 
\begin{proof} We first explain why $\chi _{\Gamma, \frac{d}{dt}} (y) = a$ is strongly minimal; it suffices to show that over some differentially closed field which contains the coefficients of the equation, that every differentially constructible set is finite or cofinite. Using properties of differentially closed field, one can find in $K$ $\tilde a$ as above.

By Lemma \ref{varchangelem}, $K$ is a $\partial$-differential field and the sets $\chi _{\Gamma, \frac{d}{dt}} (y) = a$ and $\chi _{\Gamma, \partial}  (y) =0$ coincide. Now strong minimality follows by Theorem \ref{juststrmin} and the fact that $\frac{d}{dt}$-differentially constructible sets are $\partial$-differentially constructible (over $K$). 

Algebraic dependencies among elements of the set $\chi _{\Gamma, \frac{d}{dt}} (y) = a$ give algebraic dependencies among elements of the set $\chi _{\Gamma, \partial}  (y) =0$, and thus by Theorem \ref{ALW1} must be given by $\Gamma$-special polynomials. 
\end{proof} 

The final piece of our analysis of the fibers of $\chi$ shows that there are no algebraic relations between different fibers.

\begin{thm} \label{fiberfullstatement} For $a \neq b$, the strongly minimal sets defined by $\chi _{\Gamma, \frac{d}{dt}} (y) = a$ and by $\chi _{\Gamma, \frac{d}{dt}} (y) = b$ are orthogonal.
\end{thm} 

Theorem \ref{fiberfullstatement} is more general than \cite{FreSca} Theorem 5.4, but the proof there cannot be adapted to the case of non-arithmetic fuchsian groups.

\begin{proof} Throughout, we respectively use $\mathscr M (U)$ and $\mathbb D(p,r)$  for the field of meromorphic functions on a domain $U \subset \mathbb C$, and the open complex disk centered at $p\in \mathbb C$ with radius r.  As both $\chi _{\Gamma, \frac{d}{dt}} ^{-1} (a)$ and $\chi _{\Gamma, \frac{d}{dt}} ^{-1} (b)$ are strongly minimal and geometrically trivial, if $\chi _{\Gamma, \frac{d}{dt}} ^{-1} (a) \not \perp \chi _{\Gamma, \frac{d}{dt}} ^{-1} (b)$, then there is a finite-to-finite correspondence between the sets, defined over $\m Q \langle a , b \rangle.$ Using Seidenberg's embedding theorem, we regard $a,b$ as meromorphic functions on a domain $U \subset \mathbb C$. Recall that $\tilde a$ denotes a meromorphic function such that $S_{\frac{d}{dt}}(\tilde a) = a$. The function $\tilde b$ is defined similarly. 

Using the holomorphic inverse function theorem, we claim that without loss of generality, it is enough to prove the result for the case $a=0$. Indeed, since $j_{\Gamma}(\tilde a(t))$ is interalgebraic with $j_{\Gamma}(g\tilde b(t))$ for some $g\in GL_2(\mathbb C)$, we have that $j_{\Gamma}(t)$ is interalgebraic with $j_{\Gamma}(g\tilde b({\tilde a^{-1}(t}))$ (since $\tilde b$ is defined up to composition with linear fractional transformations, we can assume that there is a common regular point for $\tilde a$ and $\tilde b$ and work locally around this point). Letting $\tilde c= \tilde b\circ\tilde a^{-1}$ and $c=S_{\frac{d}{dt}}(\tilde c)$, we see that $\chi _{\Gamma, \frac{d}{dt}} ^{-1} (0) \not \perp \chi _{\Gamma, \frac{d}{dt}} ^{-1} (c)$ and by geometric triviality this occurs over $\m Q \langle c \rangle$.

So we assume that $a=0$. Let $p$ be a regular point for $\tilde b(t)$ and let $\mathbb D_1 = \mathbb D(p,\epsilon)$ be a disc of regular points of $\tilde b(t)$. Also let $\gamma$ be a linear fractional transformation sending $\mathbb D_2 =\mathbb D(p,\frac12 \epsilon)$ to $\mathbb H$. 

Since $\chi _{\Gamma, \frac{d}{dt}} ^{-1} (0) \not \perp \chi _{\Gamma, \frac{d}{dt}} ^{-1} (b)$, we have that for some $g\in GL_2(\mathbb C)$, 
the solution $j_{\Gamma}(g\tilde b(t))$ is algebraic over 
$ \mathbb Q \langle b , j_{\Gamma}(\gamma t) \rangle \subset \mathscr{M}(\mathbb D_1)(j_{\Gamma}\circ\gamma, j_{\Gamma}'\circ\gamma, j_{\Gamma}''\circ\gamma) \subset \mathscr{M}(\mathbb D_2)$. 
But notice that for any domain $U$ such that $\mathbb D_2 \subseteq U \subseteq \mathbb D_1$, if  $j_{\Gamma}(g\tilde b(t))$ is algebraic over $\mathscr{M}(U)$, then $j_{\Gamma}(\gamma t)$ will also be algebraic over  $\mathscr{M}(U)$. 
This follows from the fact that $\mathscr{M}(\mathbb D_1)\subseteq{}\mathscr{M}(U)$, and $j_{\Gamma}(g\tilde b(t))$ is interalgebraic with $j_{\Gamma}(\gamma t)$ over $\mathbb{Q}\gen{b}\subset\mathscr{M}(\mathbb D_1)$. But $j_{\Gamma}(t)$ cannot be extended algebraically on a neighborhood of $\mathbb H$, hence $U = \mathbb D_2$. 
 
The disc $\mathbb D_2$ is thus the maximal among domains $U$ such that $j_{\Gamma}(g\tilde b(t))$ is algebraic over $\mathscr{M}(U)$. But such a domain satisfies $g\tilde b(\mathbb D_2) = \mathbb H$, that is the image of $\mathbb D_2$ by the regular holomorphic map  $\tilde b$ is the disc $g^{-1}\mathbb H$. A corollary of Schwarz's lemma gives that biholomorphisms from a disc to a disc are restrictions of homographies. 
Hence $\tilde b$ is an homography $h\in PSL_2(\mathbb C)$ and so $b=0$.
\end{proof} 

We can finally turn to the proof of the Ax-Lindemann-Weierstrass Theorem \ref{fullALW}. 

\begin{proof}[Proof of Theorem \ref{fullALW}] Recall that $V \subset \m A^n$ and for each $i=1, \ldots , n$, the variety $V$ is assumed to project dominantly onto $\m A^1$ under projection to the $i^{th}$ coordinate. Thus, the $i^{th}$ coordinate function is nonconstant, and it is possible to equip the field generated by the $i^{th}$ coordinate functions with various differential structures, which will be essential to the technique in our proof.  

\begin{lem} \label{nonzeroder} There is a derivation $\delta$ on $\m C(V)$ such that for each of the coordinate functions $t_i$ for $i=1, \ldots , n$, $\delta (t_i) \neq 0.$ 
\end{lem}
\begin{proof}

Let $z_1, \ldots, z_k$ be a transcendence basis of $\mathbb{C}(V)$ over $\mathbb C$ and $\alpha_1, \ldots , \alpha_k$ be $\mathbb Q$-linearly independent complex numbers. As $\mathbb C (V)$ is an algebraic extension of $\mathbb C (z_1,\ldots, z_k)$ the derivation $ \delta = \sum_i \alpha_i z_i \frac{\partial}{\partial z_i}$ extends a derivation of $\mathbb C(V)$ and the field of constants in $\mathbb C(V)$ is an algebraic extension of the field of constant in $\mathbb C (z_1,\ldots, z_k)$. The latter is $\mathbb C$.
As the projection of $V$ on the $i^{th}$ coordinate is dominant, $\delta (t_i)\not = 0$.
\end{proof}

The transcendence degree over $\m C(V)$ of the $3n$ functions \[j_{\Gamma}(t_1),j'_{\Gamma}(t_1),j''_{\Gamma}(t_1)\ldots,j_{\Gamma}(t_n),j'_{\Gamma}(t_n),j''_{\Gamma}(t_n)\]
is identical to that of the $3n$ functions
\[j_{\Gamma}(t_1),\delta (j_{\Gamma}(t_1)) , \delta ^2 (j_{\Gamma}(t_1)), \ldots, j_{\Gamma}(t_n), \delta (j_{\Gamma}(t_n)), \delta ^2 (j_{\Gamma}(t_n). )\]

Now, for any $t_i$, since $j_\Gamma (t_i)$ is not an algebraic function, it follows by strong minimality that $j_\Gamma(t_i)$ is a generic solution to a $\delta$-differential equation of the form $\chi_{\Gamma, \delta} (y) = a_i$ with $a_i = S_{\delta}(t_i) \in \m C(V)$.\\

If the $3n$ functions are not algebraically independent, then there exist $i,j$ such that the functions 
$$j (t_i ), \delta ( j(t_i)) , \delta^2 ( j(t_i )) , j (t_j ), \delta ( j (t_j)) , \delta^2 ( j(t_j )) $$ 
are algebraically dependent over $K$, the $\delta$-field extension of $\m C(V)$ generated by $j(t_k )$ for those $k$ in some subset of $\{1, \ldots , n \} \setminus \{i,j \}$. Moreover one can choose\footnote{Fix a subset of the coordinates such that there is an algebraic dependence as described above. Then there is some minimal such set. Picking $i,j$ to be any two coordinates of this minimal set, the subset is the collection of coordinates in the remainder of the minimal set.} 
$K$ such that $j(t_i)$ and $j(t_j)$ are not algebraic over $K$.

But then by strong minimality of the equations $\chi_{\Gamma, \delta} (y) = a_i$ and $\chi_{\Gamma, \delta} (y) = a_j$ (Theorem \ref{Strminfiber}), there is a finite-to-finite correspondence between $\chi_{\Gamma, \delta} (y) = a_i$ and $\chi_{\Gamma, \delta} (y) = a_j$ defined over $K$. By Theorem \ref{fiberfullstatement}, it must be that $a_i = a_j$ and $t_i $ and $t_j$ are $\Gamma$-geodesically dependent. A contradiction. 
\end{proof}

\subsection{Orthogonality and commutators}
In this section, we analyze the algebraic relations between solutions of 
\begin{eqnarray} 
\label{firstfiber} S_{\frac{d}{dt}}(y) +(y')^2\cdot R_{j_{\Gamma_1}}(y) =0 \\ 
\label{secondfiber} S_{\frac{d}{dt}}(y) +(y')^2\cdot R_{j_{\Gamma_2}}(y) =0 
\end{eqnarray} 
when $\Gamma_1$ is not necessarily commensurable with $\Gamma_2$. If $\Gamma_1$ is commensurable with $\Gamma_2$, then it is well known that $j_{\Gamma_1}$ is interalgebraic with $j_{\Gamma_2}$ over $\m C$. Moreover this is not the whole story: we say that $\Gamma_1$ is commensurable with $\Gamma_2$ in {\it the wide sense} if $\Gamma_1$ is commensurable to some conjugate of $\Gamma_2$. When such is the case and $\Gamma_1$ is commensurable with $g^{-1}\Gamma_2g$ then again one has that $j_{\Gamma_1}$ is interalgebraic with $j_{\Gamma_2}\circ g$ over $\m C$.

Notice that if $\Gamma_1$ is commensurable with $\Gamma_2$ in the wide sense, then $\text{Comm}(\Gamma_1)$ is conjugate to $\text{Comm}(\Gamma_2)$.  



\begin{thm} \label{GeneralcomplexparametersTHM} Suppose that $\Gamma_1$ is not commensurable with $\Gamma_2$ in the wide sense. Then the sets defined by equations \ref{firstfiber} and \ref{secondfiber} are orthogonal. In particular, for any differential field $K$  \begin{eqnarray*} {\rm tr.deg.} _K K\left( j_{\Gamma_1} (t_1), j'_{\Gamma_1}(t_1), j''_{\Gamma_1} (t_1) ,  j_{\Gamma_2} (t_2), j'_{\Gamma_2}(t_2), j''_{\Gamma_2} (t_2) \right) = \\ {\rm tr.deg.}_K K\left( j_{\Gamma_1} (t_1), j'_{\Gamma_1}(t_1), j''_{\Gamma_1} (t_1) \right)+ {\rm tr.deg.}_K K\left( j_{\Gamma_2} (t_2),  j'_{\Gamma_2}(t_2), j''_{\Gamma_2} (t_2) \right) .
\end{eqnarray*}
\end{thm}

\begin{proof} 
Let $X_{\Gamma_1}$ and $X_{ \Gamma _2 }$ be the set defined by equations \ref{firstfiber} and \ref{secondfiber} respectively. Assume for contradiction that $X_{\Gamma_1} \not \perp X_{ \Gamma _2 }$. Since $X_{\Gamma_1}$ and $X_{ \Gamma _2 }$ are \emph{trivial} strongly minimal sets, we have that nonorthogonality is witnessed over $\m C$ ({\it i.e.,} the sets are non weakly orthogonal). So for any solution $y_1\in X_{\Gamma_1}$ there is a solution $y_2\in X_{\Gamma_2}$ such that $y_1\in \m C\gen{y_2}^{alg}$. By invoking Fact \ref{nishioka1}, we have that $j _{\Gamma_1}(t)\in \m C\gen{j _{\Gamma_2}(gt)}^{alg}$ for some $g\in GL_2(\m C)$. Let us write \[P(j _{\Gamma_1}(t),j _{\Gamma_2}(gt),j _{\Gamma_2}'(gt),j _{\Gamma_2}''(gt),t)=0\]
for this algebraic relation over $\m C$. For any $\gamma_1\in \Gamma_1$, using the fact that $j _{\Gamma_1}(\gamma_1 t)=j _{\Gamma_1}(t)$, we have that
\[P(j _{\Gamma_1}(t),j _{\Gamma_2}(g\gamma_1 t),j _{\Gamma_2}'(g\gamma_1 t),j _{\Gamma_2}''(g\gamma_1 t),\gamma_1 t)=0.\]
So this implies that for any $\gamma_1\in \Gamma_1$, we get that $j _{\Gamma_1}(t)\in \m C\gen{j _{\Gamma_2}(g\gamma_1 t)}^{alg}$. In particular $\m C\gen{j _{\Gamma_2}(gt)}^{alg}=\m C\gen{j _{\Gamma_2}(g\gamma_1 t)}^{alg}$ for all $\gamma_1\in \Gamma_1$. By Theorem \ref{ALW1}, it must be the case that $g\gamma_1 g^{-1}\in \text{Comm}(\Gamma_2)$ for all $\gamma_1\in \Gamma_1$, that is it must be that $g\Gamma_1g^{-1}\subseteq \text{Comm}(\Gamma_2)$. 

Now, to get our contradiction, we consider three cases (without loss of generality): 

\begin{enumerate} 

\item Assume $\Gamma_1$ is arithmetic and $\Gamma_2$ is nonarithmetic. In this case, $\chi _ {\Gamma_1 } $ is not $\aleph _0 $-categorical, while $\chi _{\Gamma_2}$ is $\aleph _0 $-categorical (this follows from Theorem \ref{ALW1}). This case could also be handled in a more elementary manner similar to our technique in the third case. 

\item Assume that both $\Gamma_1, \Gamma_2$ are arithmetic groups. 
We have, by the above arguments, that $g \Gamma _1 g^{-1}$ is contained in 
$\text{Comm} (\Gamma_2 ). $ 
We will be done if we show that $g \Gamma_1 g^{-1}$ and $\Gamma_2$ are commensurable in the strict sense. 
This follows by arguments of \cite[see page 4]{Mcshane}, where the following fact is shown: for any two arithmetic Fuchsian groups $G_1$ and $G_2$, if $G_1 $ is contained in the commensurator of $G_2$ then $G_1 $ and $G_2$ are commensurable in the strict sense. 

\item Assume that both $\Gamma_1$ and $\Gamma_2 $ are non-arithmetic. By the above argument, we have that $g \Gamma_1 g^{-1} \leq \text{Comm} (\Gamma_2)$ for some $g \in GL_2( \m C). $ 
By a symmetric argument, we have some $h \in GL_2( \m C)$ such that $h \Gamma_2 h^{-1} \leq \text{Comm} (\Gamma_1)$. 
Replacing one of $\Gamma_i$ with a suitable conjugate, we may assume that $\Gamma_1 \leq \text{Comm} (\Gamma_2)$ and $\Gamma_2 \leq \text{Comm} (\Gamma_1)$. From this, we will show that $\Gamma_1 $ and $\Gamma_2 $ are commensurable. By Margulis' Theorem, $\Gamma_i $ is finite index in $\text{Comm} (\Gamma_i ). $ We need only show that $\Gamma_2 $ is finite index in $\text{Comm} ( \Gamma_1). $ 

We have that $\Gamma_1$ is contained in $\text{Comm} ( \Gamma _2), $ $\Gamma_1$ contains only finitely many left coset representatives of $\Gamma_2.$ Since $\Gamma_1$ is finite index in its own commensurator, the conclusion follows. 

\end{enumerate} 
\end{proof} 

\begin{rem}
The following stronger result should hold: The sets defined by $\chi _{\Gamma_1, \frac{d}{dt}} (y) = a_1$ and by $\chi _{\Gamma_2, \frac{d}{dt}} (y) = a_2$ are orthogonal if $\Gamma_1$ is not commensurable with $\Gamma_2$ in the wide sense. However, we have not been able to prove it yet.
\end{rem}

\section{Effective finiteness results around the Andr\'e-Pink conjecture}
The Andr\'e-Pink conjecture predicts that when $W$ is an algebraic subvariety of a Shimura variety and $S$ is a Hecke orbit, if $W \cap S$ is Zariski dense in $W$, then $W$ is weakly special. For details, definitions, and proofs of certain special cases of the conjecture see \cite{Orr1, Orr2, Gao}. 

In the setting of the present paper the conjecture concerns the intersection of an algebraic variety $W \subset \m A^n$ with the image, under $j_\Gamma$ applied to each coordinate, of the orbit under $\text{Comm} ( \Gamma )^n$ of some point in $\bar a \in \m H$. 

Given a Fuchsian group $\Gamma$ and a point $a \in \m C$, we denote, by $\text{Iso}_\Gamma (a),$ the collection of points $b \in \m C$ such that $P(a,b)=0$ for some $\Gamma$-special polynomial $P.$ Equivalently, for some (all) $\tilde a , \tilde b \in \m H$ such that $j_{\Gamma} (\tilde a) = a$ and $j_{\Gamma}(\tilde b)=b$, there is $\gamma \in \text{Comm} (\Gamma)$ such that $\gamma \tilde a = \tilde b.$ 

Given a Fuchsian group $\Gamma$ 
and a point $\bar a = (a_1, \ldots , a_n ) \in \m A^n(\m C)$, 
let $\text{Iso}_{\Gamma} (\bar a) $ 
denote the product of the orbits of the points $a_1, \ldots , a_n$ under $\Gamma$-special polynomials, that is $$\text{Iso}_{\Gamma } (\bar a)  = \prod_{i=1} ^n \text{Iso}_{\Gamma} (a_i).$$

We call a polynomial $p (x_1, \ldots , x_n)$ \emph{$(\Gamma)$-$(a_1, \ldots , a_n)$-special} if 
\begin{enumerate} 
\item $p(\bar x)= x_i - b_i$ where $b_i \in \text{Iso}_\Gamma (a),$ \emph{or} 
\item For some $i,j$, $\text{Iso}_{\Gamma } (a_i) = \text{Iso}_{\Gamma } (a_j)$, and $p(\bar x)$ is a $\text{Comm}(\Gamma)$-special polynomial in $x_i,x_j.$ 
\end{enumerate} 

An irreducible subvariety of $\m C^n$ will be called \emph{$(\Gamma)$-$(a_1, \ldots , a_n)$-special} if it is given by a finite conjunction of $(\Gamma  )$-$(a_1, \ldots , a_n)$-special polynomials. If an irreducible variety $V$ is $(\Gamma)$-$(a_1, \ldots , a_n)$-special, then it follows that $V$ has a Zariski dense set of points from $Iso_\Gamma (\bar a).$ Our first result of this section shows that the converse holds, at least when $\bar a$ is a tuple of transcendental numbers (perhaps with algebraic relations between them). 

\begin{thm} \label{unlikely1} Fix a complex algebraic variety $V \subset \m A^n(\m C)$, a genus zero Fuchsian group $\Gamma$ of the first kind, and a point $\bar a = (a_1, \ldots , a_n) \in \m A^n(\m C)$ such that for all but at most one $i \in \{1,\ldots , n\}$, $a_i \notin \m Q^{alg}$. Then $\overline{V \cap {\rm Iso}_{\Gamma} (\bar a) }^{Zar} $ is a finite union of $(\Gamma)$-$(a_1, \ldots , a_n)$-special varieties. 
\end{thm}

\begin{proof} 
The (perhaps reducible) variety $\overline{V \cap \text{Iso}_{\Gamma} (\bar a) }^{Zar} $ consists of finitely many components $W_1, \ldots , W_k,$ and so we need only show that the varieties $W_i$ are $(\Gamma )$-$(a_1, \ldots , a_n)$-special. Working component by component, it suffices to show that for an arbitrary irreducible variety $V$, if $\text{Iso}_{\Gamma} (\bar a)$ is Zariski dense in $V$, then $V$ is $(\Gamma )$-$(a_1, \ldots , a_n)$-special. 

Without loss of generality, assume that all of the coordinates of $\bar a$, except perhaps $a_1$, are transcendental over $\m Q$. We also assume $a_1 \in \m Q^{alg}$ without loss of generality - otherwise just ignore arguments about this coordinate in the proof.

Embed $\m Q (a_2, \ldots , a_n)$ into the field of meromorphic functions on some connected subset of $\m H$ such that $a_2, \ldots , a_n$ are non-constant. 

Let $\tilde a_2, \ldots , \tilde a_n$ be as in the proof of Theorem \ref{Strminfiber} - that is, $j_{\Gamma_i}(\tilde a_i) = a_i$ for $i=2, \ldots , n.$ In the differential closure $K$ of the field generated by the $a_i$ over $\m Q$ we have, by Theorem \ref{Strminfiber}, that 
$$\{ x \in K \, | \, \chi_{\Gamma} (x) = \chi_{\Gamma} (a_i) \} = \text{Iso}_{\Gamma} (a_i),$$ so $\chi _{\Gamma} (a_i) = \chi_{\Gamma} (a_j)$ if and only if $\text{Iso}_{\Gamma} (a_i) = \text{Iso} _{\Gamma} (a_j).$ 

Consider the collection of $i \in \{1,\ldots , n\} $ such that $V$ projects dominantly onto the coordinate corresponding to $x_i$. Then if $\text{Iso}_{\Gamma} (\bar a)$ is dense in $V$, and we let $b_2, \ldots , b_n$ be a collection of generic solutions of $\chi_{\Gamma} (b_i) = \chi_{\Gamma } (a_i)$ and let $b_1$ be a generic constant, we have that the tuple $\bar b$ is dependent over $\m C$, but as $b_2, \ldots , b_n$ satisfy equations which are strongly minimal and trivial, it must be that two of the coordinates are nonorthogonal. But now we are done, since all instances of nonorthogonality are given by Theorem \ref{Strminfiber}, since none of the coordinates $2, \ldots , n$ can be nonorthogonal to $b_1$, a constant.
\end{proof}

\begin{rem}\label{PrevRem} The assumption in Theorem \ref{unlikely1} that all but at most one of the elements in the tuple $\bar a$ are transcendental is an inherent restriction of the method we employ, which is similar to the technique employed in various applications of differential algebra to diophantine problems. We replace an arithmetic (discrete) object by the solution to a system of differential equations, then reduce the general case to an analytic statement using a strong version of Seidenberg's embedding theorem. Generally speaking, the technique works when the discrete set satisfies some interesting differential equation, which one is able to understand. But the only derivation on $\m Q^{alg}$ is the trivial one, and so such a coordinate can not . For other instances of applications of this general idea, see \cite{newML, ML1, ML2, Buium}. 

It would be interesting to see if the methods here might be combined with methods solving other special cases of the conjecture (e.g. \cite{Orr1}) to remove the transcendence restrictions of Theorem \ref{unlikely1}. 
\end{rem}

\begin{rem} The technique by which we prove Theorem \ref{unlikely1} has natural limitations described in Remark \ref{PrevRem}, but it also has an interesting natural advantage over other techniques. Because we replace an arithmetic object, whose definition is very \emph{non-uniform}, with a differential algebraic variety, results from differential algebraic geometry can be used to give effective bounds the degree of the Zariski-closure of the solutions set. 
\end{rem}

\par A general purpose Bezout-type theorem for algebraic differential equations (generalizing a theorem of Hrushovski and Pillay) was established in \cite{FreSan}. In what follows, $\tau _{\ell} \m A ^n$ denotes the $\ell^{th}$-prolongation space of $\m A ^n$ and for a differential field $K$, we define $$(X,S \setminus T )^ \sharp(K)=\{a\in X(K):(a,a',\ldots,a^{(\ell)})\in S \setminus T(K)\}.$$
\begin{thm} \label{JO} Let $X$ be a closed subvariety of $\m A ^n$, with $\dim (X) = m$, and let $S, T$ be closed subvarieties (not necessarily irreducible) of $\tau _{\ell} \m A ^n$ for some ${\ell}\in \mathbb N$. Then the degree of the Zariski closure of $(X,S \setminus T )^ \sharp(\mathbb C)$ is at most $\deg (X) ^ {{\ell} 2^{m{\ell}}} \deg (S)^ {2^{m{\ell}}-1}$. In particular, if $(X,S \setminus T )^ \sharp(\mathbb C)$ is a finite set, this expression bounds the number of points in that set.
\end{thm} 
Next, we aim to put our differential relations in a form such that we may apply Theorem \ref{JO}. Recall our Schwarzian differential equation: 
 
\begin{equation}\tag{$\star$} \label{stareqn1}
S_{\frac{d}{dt}}(y) +(y')^2\cdot R_{j_{\Gamma}}(y) =0
\end{equation}

where $S_{\frac{d}{dt}}(y)=\frac{y'''}{y'} -\frac{3}{2}\left(\frac{y''}{y'}\right)^2$ denotes the Schwarzian derivative ($'=\frac{d}{dt}$) and $R_{j_{\Gamma}}\in\mathbb{C}(y)$ depends on the choice of $j_{\Gamma}$. For the purposes of this section, all that matters is the degree of the rational function $R_{j_{\Gamma}}$ (the coefficients, which are complex numbers, will not be important in stating or proving our results). If the $\Gamma$-action on $\mathbb{H}$ has a fundamental half domain given by a $r$-sided polygon $P$ (note that this is the case for any Fuchsian group of the first kind as $r$ is equal to the number of generators of $\Gamma$ \cite{Katok}), then 
\[R_{j_{\Gamma}}(y)=\frac{1}{2}\sum_{i=1}^{r}{\frac{1-\alpha_i^2}{(y-a_i)^2}}+\sum_{i=1}^{r}{\frac{A_i}{y-a_i}},\]
where the coefficients are complex numbers depending on specific characteristics of the domain. The crucial point for our results is that the degree of $R_{j_\Gamma}$ (by which we mean the maximum of the degree of the numerator and the denominator) is given by $2r$ where $r$ is the number of generators of $\Gamma$. 

Clearing the denominator of the rational function and the Schwarzian in equation (\ref{stareqn1}), we obtain:

\begin{multline}\tag{$Q(\star)$} \label{stareqn2}
0 = (y''' y' - \frac{3}{2} (y'')^2) \prod_{i=1}^ r (y-\alpha_i)^2 \;\;+\\ (y')^4 \left(  \frac{1}{2} \sum_{i=1}^r \left( (1-\alpha_i) \prod_{j \in [r], \, j \neq i} (y-a_i)^2 \right)  + \sum _{i=1}^r \left( A_i (y-a_i)\prod_{j \in [r], j \neq i} (y-a_j)^2\right)  \right) 
\end{multline}
As a polynomial, the previous equation has degree $2r+2.$

\begin{thm} \label{unlikely1e} Fix a complex algebraic variety $V \subset \m A^n(\m C)$, a genus zero Fuchsian group $\Gamma$ of the first kind, and a point $\bar a = (a_1, \ldots , a_n) \in \m A^n(\m C)$ such that for all $i \in \{1,\ldots , n\}$, $a_i \notin \m Q^{alg}$. Then $\overline{V \cap {\rm Iso}_{\Gamma} (\bar a) }^{Zar} $ is a finite union of $(\Gamma )$-$(a_1, \ldots , a_n)$-special varieties, and the sum of the degrees of the varieties in this union is at most $$((2r+2)^n \cdot \deg( V) )^{2^{3n}-1}.$$
\end{thm}

\begin{proof} We need only put the equations appearing in Theorem \ref{unlikely1} in a form suitable to apply Theorem \ref{JO}. We can write the Schwarzian differential equations as $\nabla_3 ^{-1}  (S)$ on each coordinate, where $S$ is the locus of (\ref{stareqn2}) in $\tau _3 (\m A^1).$ On each coordinate, this equation has degree $2r$, so the intersection of these relations with $V$ is a variety in $\tau_3 (\m A^n)$ of degree at most $(2r+2)^n \deg (V).$ Now the degree bound follows from Theorem \ref{JO} with $X= \m A^n$, $l=3$, and $V$ as given above. 
\end{proof} 

\begin{rem} One can also establish (by the same means as in the previous proof) a version of Theorem \ref{unlikely1e} with one coordinate algebraic rather than transcendental (the bound is slightly better in this case). The bounds of Theorem \ref{unlikely1e} can also be improved (using more elaborate arguments) by applying the results of \cite{Galfin}, a process carried out in \cite{Galfin} in the case that $\Gamma$ is the modular group. 
\end{rem}

\appendix
\section{Strong minimality for the special case of triangle groups}\label{TriangleEx}\label{Appendix}

In this appendix, we discuss an alternate method of proving strong minimality of the Schwarzian equation in the special case of triangle groups. As before, we assume that $\Gamma$ is a Fuchsian group of first kind and of genus zero. The group $\Gamma$ is said to be a Fuchsian triangle group of type $(k,l,m)$ if its signature is $(0;k,l,m)$ (see Section \ref{basic}). We will without loss of generality always assume that $2\leq k\leq l\leq m\leq \infty$. We write $\Gamma_{(k,l,m)}$ for the Fuchsian triangle group of type $(k,l,m)$.

The fundamental domain in $\mathbb{H}$ of $\Gamma_{(k,l,m)}$ is the union of a hyperbolic triangle with angles $\frac{\pi}{k}$, $\frac{\pi}{l}$ and $\frac{\pi}{m}$ at the  vertices $v_k$, $v_l$ and $v_m$ respectively, together with its image via hyperbolic reflection of one side connecting the vertices. Notice that since $k,l,m$ relates to the angle of an hyperbolic triangle, if $\Gamma_{(k,l,m)}$ is a triangle group then
\[\frac{1}{k}+\frac{1}{l}+\frac{1}{m}<1.\]
Also, the vertices $v_k$, $v_l$ and $v_m$ are the fixed points of the generators $g_1$, $g_2$ and $g_3$ respectively.
\begin{defn}\label{Assump}
The function $j_{(k,l,m)}$ will denote the (unique) Hauptmodul $\Gamma_{(k,l,m)}\setminus \mathbb{H}_{\Gamma_{(k,l,m)}}\rightarrow{\bf P}^1(\mathbb{C})$ sending $v_k$, $v_l$, $v_m$ to $1$, $0$, $\infty$ respectively.
\end{defn}
With this definition (cf. \cite[Chapter 5]{AbloFoka}) we have that $j_{(k,l,m)}$ satisfies the Schwarzian equation $(\star)$ with
\begin{eqnarray} \label{trianglerat} R_{j_{(k,l,m)}}(y)=\frac{1-l^{-2}}{y^2}+\frac{1-k^{-2}}{(y-1)^2}+\frac{k^{-2}+l^{-2}-m^{-2}-1}{y(y-1)}. \end{eqnarray}
Notice that with Definition \ref{Assump}, the Hauptmodul $j_{(2,3,\infty)}$ for $PSL_2(\mathbb{Z})$ is not the classical $j$-funtion. Rather, one has that $j=1728j_{(2,3,\infty)}$ (see Example \ref{jfunction}).
\par Finally let us mention that there is a full classification, up to $PSL_2(\mathbb{R})$-conjugation, of the arithmetic triangle groups
\begin{fct}
Up to $PSL_2(\mathbb{R})$-conjugation, there are finitely many arithmetic triangle groups; 76 cocompact and 9 non-cocompact \cite{Take}. Among these, there are 19 distinct commensurability classes represented \cite{take2}. 
\end{fct}

In the special case of triangle groups, proving that the Riccati equation \ref{ric1} has no algebraic solutions (and thus establishing the strong minimality of the associated order three nonlinear Schwarzian differential equations) can be accomplished without any appeal to Picard-Vesiot theory but instead by using classical work around the hypergeometric equation. Already, in \cite[see page 601]{NishiokaII}, Nishioka shows that equation \ref{RiccatiJ} has no algebraic solutions in the case the $\Gamma$ is a cocompact triangle group (which corresponds to the case that none of $k,l,m$ are $\infty$). Hence Condition \ref{Ric} and thus Theorem \ref{juststrmin} holds in the case of cocompact triangle groups.  We will, via a very similar argument, show the same result holds in the case that $\Gamma$ is not cocompact. To emphasize, these results are a special case of our general result on Fuchsian groups, but we feel their inclusion is worthwhile in part because the method, which deals more directly with the order two linear equation \ref{o2} and Riccati equation \ref{RiccatiJ}, might generalize to Schwarzian equations of the form of equation (\ref{stareqn}') which do not necessarily come from a group action of $\Gamma$ on $\m H$. This restriction appears to be more inherent in our main approach of the previous section. 

Let \begin{eqnarray} 
\lambda = \frac{1}{l} \\
\mu = \frac{1}{k} \\
\nu = \frac{1}{m} 
\end{eqnarray} 
where the integers $2 \leq k \leq l \leq m \leq \infty$ are as above. We have already seen $\lambda + \mu + \nu <1$. 
Now let $\alpha$, $\beta$ and $\gamma$ be any complex numbers such that, $\lambda = 1- \gamma $, $\mu = \gamma - \alpha - \beta $, and $\nu = \alpha -\beta $.

Now, we know that the second order equation \ref{o2} corresponding to equation (\ref{stareqn}) with rational function \ref{trianglerat} (equation (5) of \cite{NishiokaII}) is reducible if and only if one of $\alpha, \beta , \gamma  - \alpha, \gamma - \beta $ is an integer. Since \cite{NishiokaII} covers the cocompact case, we can assume without loss of generality that $m = \infty$, equivalently $\nu=0$. Thus, in the above notation, $\alpha = \beta $. Now, $$\alpha = \frac{1 - \frac{1}{l} - \frac{1}{k}}{2}.$$ In this case, by the triangle requirement, 
$\frac{1}{l} + \frac{1}{k} <1$, so $\alpha$ is never an integer. 

Further, we have $$\gamma - \alpha = \frac{1-\frac{1}{l} + \frac{1}{k}}{2}.$$ This quantity is never an integer, since $\frac{1}{l} + \frac{1}{k} <1$. Thus, in the non-cocompact case, we have that the corresponding equation \ref{o2} is always irreducible, which, by the correspondence explained in \ref{Liouville} implies that there are no rational solutions to equation \ref{ric1} in this case. 

Now, under the assumption of irreducibility of equation \ref{o2}, we have that there is an algebraic (but irrational) solution of \ref{ric1} if and only if two of $\lambda - \frac{1}{2} , \mu - \frac{1}{2} , \nu - \frac{1}{2} $ are integers \cite[pages 96-100]{Matsuda}. This is impossible for any triangle group as at most one of these is an integers as long as $\lambda +\mu + \nu <1$. 

Thus, we have shown, in a more direct way, that Condition \ref{Ric} and thus Theorem \ref{juststrmin} also holds in the case of non-cocompact triangle groups. 

\begin{rem}
At first glance the above arguments only seem to show that the differential equations for the unformizers $j_{(k,l,m)}$ are strongly minimal. However, all other uniformizers are rational functions (over $\mathbb C$) of the $j_{(k,l,m)}$'s. From this, strong  minimality follows for the other equations as well.
\end{rem}

\end{document}